\documentclass[12pt,reqno]{amsart}

\usepackage{wasysym, amsmath, amssymb,graphicx,amsthm,latexsym, amsfonts, enumitem, mathtools, tensor}
\usepackage{hyperref}
\usepackage{xcolor}
\usepackage[all, color]{xy}
\usepackage{color}
\usepackage{float}
\usepackage{tikz}
\usetikzlibrary{arrows,decorations.pathmorphing,decorations.pathreplacing,positioning,shapes.geometric,shapes.misc,decorations.markings,decorations.fractals,calc,patterns}

\DeclareMathOperator{\Hom}{Hom}
\DeclareMathOperator{\End}{End}

\DeclareMathOperator{\im}{Im}

\DeclareMathOperator{\Coker}{Coker}

\DeclareMathOperator{\mmod}{mod}

\DeclareMathOperator{\Tr}{Tr}

\DeclareMathOperator{\rad}{rad}
\DeclareMathOperator{\Ext}{Ext}

\DeclareMathOperator{\add}{add}

\theoremstyle{plain}
\newtheorem{theorem}{Theorem}[section]
\newtheorem*{theorem*}{Theorem}

\theoremstyle{definition}
\newtheorem{defn}[theorem]{Definition}

\newtheorem{remark}[theorem]{Remark}

\newtheorem{lemma}[theorem]{Lemma}
\newtheorem{notation}[theorem]{Notation}
\newtheorem{corollary}[theorem]{Corollary}

\newtheorem{setup}[theorem]{Setup}
\newtheorem{proposition}[theorem]{Proposition}

\date{}
\begin{document}
\setlength{\parindent}{0pt}
\setlength{\parskip}{7pt}
\title[$d$-Auslander-Reiten sequences in subcategories]{\MakeLowercase{\textit{d}}-Auslander-Reiten sequences in subcategories}
\author{Francesca Fedele}
\address{School of Mathematics, Statistics and Physics,
Newcastle University, Newcastle upon Tyne NE1 7RU, United Kingdom}
\email{F.Fedele2@newcastle.ac.uk}

\begin{abstract}
    Let $\Phi$ be a finite dimensional algebra over a field $k$. Kleiner described the Auslander-Reiten sequences in a precovering extension closed subcategory $\mathcal{X}\subseteq\mmod\Phi$. If $X\in\mathcal{X}$ is an indecomposable such that $\Ext_{\Phi}^1(X,\mathcal{X})\neq 0$ and $\zeta X$ is the unique indecomposable direct summand of the $\mathcal{X}$-cover $g:Y\rightarrow D\Tr X$ such that $\Ext_{\Phi}^1(X,\zeta X)\neq 0$, then there is an Auslander-Reiten sequence in $\mathcal{X}$ of the form
    \begin{align*}
        \epsilon: 0\rightarrow \zeta X\rightarrow X'\rightarrow X\rightarrow 0.
    \end{align*}
    Moreover, when $\End_\Phi (X)$ modulo the morphisms factoring through a projective is a division ring, Kleiner proved that each non-split short exact sequence of the form
    \begin{align*}
        \delta: 0\rightarrow Y\rightarrow Y'\xrightarrow{\eta} X\rightarrow 0
    \end{align*}
    is such that $\eta$ is right almost split in $\mathcal{X}$, and the pushout of $\delta$ along $g$ gives an Auslander-Reiten sequence in $\mmod\Phi$ ending at $X$.
    
    In this paper, we give higher dimensional generalisations of this. Let $d\geq 1$ be an integer. A $d$-cluster tilting subcategory $\mathcal{F}\subseteq\mmod\Phi$ plays the role of a higher $\mmod\Phi$. Such an $\mathcal{F}$ is a $d$-abelian category, where kernels and cokernels are replaced by complexes of $d$ objects and short exact sequences by complexes of $d+2$ objects. We give higher versions of the above results for an additive ``$d$-extension closed'' subcategory $\mathcal{X}$ of $\mathcal{F}$.
\end{abstract}
\maketitle

\section{Introduction}
Let $d$ be a fixed positive integer, $k$  a field and $\Phi$ a finite dimensional $k$-algebra. Let $\mmod\Phi$ denote the category of finitely generated right $\Phi$-modules.

\subsection{Classic background (\texorpdfstring{$d=1$}{TEXT} case).}
Auslander-Reiten sequences in $\mmod\Phi$ are non-split short exact sequences that are a very useful tool to study indecomposable modules in $\mmod\Phi$ and the morphisms between them.
If $M\in\mmod\Phi$ is an indecomposable non-projective module, then there is an Auslander-Reiten sequence in $\mmod\Phi$ of the form:
\begin{align*}
    \xymatrix{
    0\ar[r]& D\Tr M\ar[r]& N\ar[r]& M\ar[r]&0,
    }
\end{align*}
where $D\Tr$ is the Auslander-Reiten translation. Then the components of the morphism $N\rightarrow M$ are all the irreducible morphisms ending at the indecomposable $M$ and the components of $D\Tr M\rightarrow N$ are all the irreducible morphisms starting at the indecomposable $D\Tr M$.
A detailed study of Auslander-Reiten sequences and their use can be found in \cite[Chapter V]{ARS} for example.

Let $\mathcal{X}\subseteq\mmod\Phi$ be a full subcategory closed under summands and extensions, in the sense that if $0\rightarrow X\rightarrow Y\rightarrow Z\rightarrow 0$ is a short exact sequence in $\mmod\Phi$ with  $X,Z\in\mathcal{X}$, then $Y\in\mathcal{X}$. Auslander and Smal{\o}  introduced the notion of almost split sequences in subcategories and, in \cite[Theorem 2.4]{AS}, showed a weaker version of the following theorem introduced later by Kleiner in \cite[Corollary 2.8]{MK}.

\textbf{Theorem A (Kleiner).}
{\em
Assume $\mathcal{X}$ is precovering in $\mmod\Phi$ and let $X$ be an indecomposable in $\mathcal{X}$.
\begin{enumerate}[label=(\alph*)]
    \item There exists a right almost split morphism $W\rightarrow X$ in $\mathcal{X}$.
    \item If $\Ext_{\Phi}^1(X,\mathcal{X})$ is non-zero, there is an Auslander-Reiten sequence in $\mathcal{X}$ of the form:
    \begin{align*}
    \xymatrix{
    0\ar[r]& \zeta X\ar[r]& X^1\ar[r]& X\ar[r]&0,
    }
    \end{align*}
    where $\zeta X$ is the unique indecomposable direct summand of the $\mathcal{X}$-cover of  $D\Tr X$ such that $\Ext_{\Phi}^1(X,\zeta X)\neq 0$.
\end{enumerate}
}

For $M\in\mmod\Phi$, let $\underline{\End}_\Phi(M)$ denote the factor ring of $\End_\Phi (M)$ modulo the ideal of morphisms $M\rightarrow M$ that factor through a projective module. Then, Auslander, Reiten and Smal\o's argument in \cite[proof of Corollary V.2.4]{ARS} can be easily modified to prove the following.

\textbf{Theorem B.}
{\em Assume $\mathcal{X}$ is precovering in $\mmod\Phi$. Let $X\in\mathcal{X}$ be an indecomposable such that $\underline{\End}_\Phi(X)$ is a division ring. For a short exact sequence of the form
\begin{align*}
    \xymatrix{
    \delta:& 0\ar[r]& \zeta X\ar[r]& X^1\ar[r]& X\ar[r]&0,
    }
\end{align*}
the following are equivalent:
\begin{enumerate}[label=(\alph*)]
    \item $\delta$ is an Auslander-Reiten sequence in $\mathcal{X}$,
    \item $\delta$ does not split.
\end{enumerate}
}

As a corollary of the above, one can prove the following result by Kleiner, see \cite[Proposition 2.10]{MK}.

\textbf{Corollary C (Kleiner).} 
{\em Assume $\mathcal{X}$ is precovering in $\mmod\Phi$. Let $g:Y\rightarrow D\Tr (X)$ be an $\mathcal{X}$-cover, where $X$ is an indecomposable in $\mathcal{X}$ with $\underline{\End}_\Phi(X)$ a division ring. Consider a non-split short exact sequence with terms in $\mathcal{X}$ of the form 
\begin{align*}	
\xymatrix{
    0\ar[r]& Y\ar[r]& Y^1\ar[r]^\eta& X\ar[r]&0.
    }
\end{align*}
Then the bottom row of the pushout diagram
\begin{align*}
    \xymatrix@!0 @R=3.5em @C=4em{
    0\ar[r]& Y\ar[r]\ar[d]& Y^1\ar[r]^\eta\ar[d]& X\ar[r]\ar@{=}[d]&0\\
    0\ar[r] & D\Tr X\ar[r]& N\ar[r]& X\ar[r]&0
    }
\end{align*}
is an Auslander-Reiten sequence in $\mmod\Phi$ and $\eta$ is right almost split in $\mathcal{X}$.
}

\subsection{This paper ($d\geq 1$ case).}
Assume now that there is a $d$-cluster tilting subcategory $\mathcal{F}\subseteq \mmod\Phi$, \textbf{i.e.} a functorially finite additive subcategory such that
\begin{align*}
    \mathcal{F}=\{ A\in\mmod\Phi\mid \Ext_{\Phi}^{1,\dots, d-1} (\mathcal{F},A)=0 \}=\{ A\in\mmod\Phi\mid \Ext_{\Phi}^{1,\dots, d-1} (A,\mathcal{F})=0 \},
\end{align*}
see \cite[Definition 1.1]{I}.
In \cite{JG}, Jasso generalised abelian categories to $d$-abelian categories: kernels and cokernels are replaced by complexes of $d$ objects, called $d$-kernels and $d$-cokernels respectively, and short exact sequences by complexes of $d+2$ objects, called $d$-exact sequences, see Definition \ref{defn_d_Jasso}. Then, $\mathcal{F}$ is a $d$-abelian category and it plays the role of a higher version of the abelian category $\mmod\Phi$. Note that for $d=1$, the only possible choice is $\mathcal{F}=\mmod\Phi$.

In \cite{IO}, Iyama generalised Auslander-Reiten sequences in $\mmod\Phi$ to $d$-Auslander Reiten sequences in $\mathcal{F}$. Moreover, he proved in \cite[Theorem 3.3.1]{IO} that if $A^{d+1}$ is an indecomposable non-projective in $\mathcal{F}$, then there exists a $d$-Auslander-Reiten sequence in $\mathcal{F}$, see Definition \ref{defn_dAR_seq} with $\mathcal{X}=\mathcal{F}$, of the form:
\begin{align*}
  \xymatrix @C=1.5em{
    0\ar[r]& D\Tr_d (A^{d+1}) \ar[r]& A^1\ar[r]& A^2\ar[r]&\cdots\ar[r]&A^{d-1}\ar[r]&A^d\ar[r]& A^{d+1}\ar[r]&0,
    }
\end{align*}
where $D\Tr_d$ is the $d$-Auslander-Reiten translation and $\Tr_d$ is as described in Definition \ref{defn_d_higher_transpose}. Let $\mathcal{X}\subseteq \mathcal{F}$ be an additive subcategory in the sense of Definition \ref{defn_additive_subcat} that is closed under $d$-extensions, see Definition \ref{defn_closed_u_ext}. We define $d$-Auslander-Reiten sequences in $\mathcal{X}$ and prove a higher version of Theorem A.

\textbf{Theorem \ref{coro_d-ARseq}.}
{\em
Assume $\mathcal{X}$ is precovering in $\mathcal{F}$ and let $X$ be an indecomposable in $\mathcal{X}$.
\begin{enumerate}[label=(\alph*)]
    \item There exists a right almost split morphism $W\rightarrow X$ in $\mathcal{X}$.
    \item If $\Ext_{\Phi}^d(X,\mathcal{X})$ is non-zero, there is a $d$-Auslander-Reiten sequence in $\mathcal{X}$ of the form:
    \begin{align*}
    \xymatrix{0\ar[r]&\sigma X\ar[r]^-{\xi^0}& X^1\ar[r]^{\xi^1}&\cdots\ar[r]& X^{d-1}\ar[r]^-{\xi^{d-1}} & X^d\ar[r]^-{\xi^d}& X\ar[r]&0,
    }
\end{align*}
where $\sigma X$ is the unique indecomposable direct summand of the $\mathcal{X}$-cover of $D\Tr_d (X)$ such that $\Ext_{\Phi}^d(X,\sigma X)\neq 0$.
\end{enumerate}
}

Let $\rad_\mathcal{X}$ denote the Jacobson radical of $\mathcal{X}$, that is the two sided ideal of $\mathcal{X}$ defined by
\begin{align*}
    \rad_\mathcal{X}(X,Y)=\{ \xi: X\rightarrow Y\mid 1_X-\eta\xi \text{ is invertible for any } \eta:Y\rightarrow X \},
\end{align*}
for all objects $X$ and $Y$ in $\mathcal{X}$. We prove a higher version of Theorem B.

\textbf{Theorem \ref{coro_ARS_gen}.}
{\em
Assume $\mathcal{X}$ is precovering in $\mathcal{F}$. Let $X$ be an indecomposable in $\mathcal{X}$ such that $\underline{\End}_\Phi(X)$ is a division ring. Let
\begin{align*}
    \xymatrix{\delta: 0\ar[r]&\sigma X\ar[r]^-{\xi^0}& X^1\ar[r]^{\xi^1}&\cdots\ar[r]& X^{d-1}\ar[r]^-{\xi^{d-1}} & X^d\ar[r]^{\xi^d}& X\ar[r]&0
    }
\end{align*}
be a  $d$-exact sequence with terms in $\mathcal{X}$ and such that $\xi^1,\dots,\,\xi^{d-1}$ are in $\rad_{\mathcal{X}}$ when $d\geq 2$. Then the following are equivalent:
\begin{enumerate}[label=(\alph*)]
    \item $\delta$ is a $d$-Auslander-Reiten sequence in $\mathcal{X}$,
    \item $\delta$ does not split.
\end{enumerate}
}

In \cite{JG}, Jasso generalised the idea of pushout to $d$-pushout of a $d$-exact sequence along a morphism from its first term, see Definition \ref{defn_dpush} and Lemma \ref{lemma_d-pushout_exists}. Then, we obtain a higher version of Corollary C as a corollary of Theorem \ref{coro_ARS_gen}.

\textbf{Corollary \ref{coro_dAR_division}.}
{\em
Assume $\mathcal{X}$ is precovering in $\mathcal{F}$. Let $g:Y\rightarrow D\Tr_d (X)$ be an $\mathcal{X}$-cover, where $X$ is an indecomposable in $\mathcal{X}$ with $\underline{\End}_\Phi(X)$ a division ring. Consider a non-split $d$-exact sequence with terms in $\mathcal{X}$ of the form:
\begin{align*}
    \xymatrix{
\epsilon:&0\ar[r] &Y\ar[r]^{\eta^0}& Y^1\ar[r]^{\eta^1}& \cdots\ar[r]&  Y^d\ar[r]^{\eta^d}& X\ar[r]&0,}
\end{align*}
where, if $d\geq 2$, we also have $\eta^1,\dots,\,\eta^{d-1}\in\rad_{\mathcal{X}}$. Consider a morphism induced by a $d$-pushout diagram:
\begin{align*}
\xymatrix{
\epsilon:\ar[d]&0\ar[r] &Y\ar[r]^{\eta^0}\ar[d]^-{g}& Y^1\ar[r]^{\eta^1}\ar[d]^-{g^1}& \cdots\ar[r]&  Y^d\ar[d]^-{g^d}\ar[r]^{\eta^d}& X\ar@{=}[d]\ar[r]&0\\
 \delta: & 0\ar[r]&D\Tr_d (X)\ar[r]_-{\alpha^0}& A^1\ar[r]_{\alpha^1}& \cdots\ar[r]& A^d\ar[r]_{\alpha^d}& X\ar[r]&0,
}
\end{align*}
where, if $d\geq 2$, we have that $\alpha^1,\dots,\,\alpha^{d-1}\in\rad_{\mathcal{F}}$. Then $\delta$ is a $d$-Auslander-Reiten sequence in $\mathcal{F}$ and $\eta^d$ is right almost split in $\mathcal{X}$.
}

We illustrate Theorem \ref{coro_d-ARseq} in the following example with $d=2$. Let $\Phi$ be the algebra defined by the following quiver with relations:
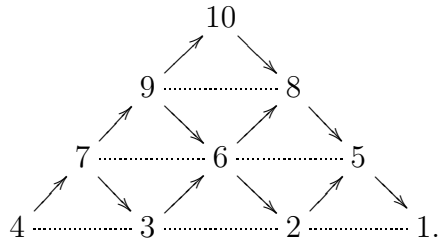
\begin{align*}
\xymatrix @C=1em@R=1em{
&&& 10\ar[rd] &&&\\
&& 9\ar[ru]\ar[rd]\ar@{..}[rr] && 8\ar[rd] &&\\
& 7 \ar[ru]\ar[rd]\ar@{..}[rr] && 6\ar[ru]\ar[rd]\ar@{..}[rr] && 5\ar[rd]&\\
4\ar[ru]\ar@{..}[rr] && 3\ar[ru]\ar@{..}[rr] && 2\ar[ru]\ar@{..}[rr] && 1.
}
\end{align*}
The Auslander-Reiten quiver of the unique $2$-cluster tilting subcategory $\mathcal{F}$ of $\mmod\Phi$ is shown in Figure \ref{fig:AR_F} on page \pageref{fig:AR_F}. Choosing a  subcategory $\mathcal{X}\subseteq \mathcal{F}$ satisfying our setup, namely add of the vertices coloured red in Figure \ref{fig:AR_F}, we use Theorem \ref{coro_d-ARseq} to describe the $2$-Auslander-Reiten sequences in $\mathcal{X}$.

The paper is organised as follows. Section \ref{section_defn} recalls the definitions of $d$-abelian and $d$-cluster tilting subcategories. Section \ref{section_3} presents some properties of $d$-pushout diagrams, $d$-exact sequences and their morphisms. Section \ref{section_4} studies $d$-Auslander-Reiten sequences in $\mathcal{X}$. Section \ref{section_Kleiner2} proves higher analogues to some of Kleiner's results from \cite[Section 2]{MK}, including Theorem \ref{coro_d-ARseq}. Section \ref{section_6} proves Theorem \ref{coro_ARS_gen} and Corollary \ref{coro_dAR_division}. Finally, Section \ref{section_7} illustrates an example of Theorem \ref{coro_d-ARseq}.

\section{Definitions of $d$-abelian categories and $d$-cluster tilting subcategories}\label{section_defn}
Let $d$ be a fixed positive integer, $k$  a field and $\Phi$ a finite dimensional $k$-algebra.
In this section we recall the definitions of $d$-abelian categories and $d$-cluster tilting subcategories of the category of finitely generated right $\Phi$-modules.

\begin{notation}\label{notation_right_modules}
Unless otherwise specified, we assume that $\Lambda$-modules of any $k$-algebra $\Lambda$ are right $\Lambda$-modules. The category of finitely generated right $\Lambda$-modules is denoted $\mmod\Lambda$ and the one of finitely generated left $\Lambda$-modules is denoted $\mmod\Lambda^{op}$.
\end{notation}

\begin{defn}[{\cite[Definitions 2.2, 2.4 and 2.9]{JG}}]\label{defn_d_Jasso}
Let $\mathcal{A}$ be an additive category.
\begin{enumerate}[label=(\alph*)]
    \item A diagram of the form
    $\xymatrix{
    A^0\ar[r]& A^1\ar[r]& A^2\ar[r]&\cdots\ar[r]&A^{d-1}\ar[r]&A^d
    }$
    is a \textit{$d$-kernel} of a morphism $\xymatrix{A^d\ar[r]& A^{d+1}}$ if 
    \begin{align*}
    \xymatrix{
    0\ar[r] & \Hom_{\mathcal{A}}(B,A^0)\ar[r]& \cdots\ar[r]&\Hom_{\mathcal{A}}(B,A^{d})\ar[r]& \Hom_{\mathcal{A}}(B,A^{d+1})
    }
    \end{align*}
    is an exact sequence for each $B$ in $\mathcal{A}$.
    \item A diagram of the form
    $\xymatrix{
    A^1\ar[r]&A^2\ar[r]&\cdots\ar[r]&A^{d-1}\ar[r]&A^d\ar[r]& A^{d+1}
    }$
   is a \textit{$d$-cokernel} of a morphism $\xymatrix{A^0\ar[r]& A^{1}}$ if 
    \begin{align*}
     \xymatrix{
    0\ar[r]&\Hom_{\mathcal{A}}(A^{d+1},B)\ar[r]&\cdots\ar[r]&\Hom_{\mathcal{A}}(A^1,B)\ar[r]& \Hom_{\mathcal{A}}(A^0,B)
    }
    \end{align*}
    is an exact sequence for each $B$ in $\mathcal{A}$.
    \item A \textit{$d$-exact sequence} is a diagram of the form:
    \begin{align*}
        \xymatrix{
    0\ar[r]&A^0\ar[r]^{\alpha^0}& A^1\ar[r]& A^2\ar[r]&\cdots\ar[r]&A^{d-1}\ar[r]&A^d\ar[r]^{\alpha^d}& A^{d+1}\ar[r]&0,
    }
    \end{align*}
    such that $\xymatrix{A^0\ar[r]^{\alpha^0}& A^1\ar[r]& A^2\ar[r]&\cdots\ar[r]&A^{d-1}\ar[r]&A^d}$ is a $d$-kernel of $\alpha^d$ and $\xymatrix{A^1\ar[r]& A^2\ar[r]&\cdots\ar[r]&A^{d-1}\ar[r]&A^d\ar[r]^{\alpha^d}& A^{d+1}}$ is a $d$-cokernel of $\alpha^0$.
    \item A \textit{morphism of $d$-exact sequences} is a chain map:
    \begin{align*}
    \xymatrix{0\ar[r]&A^0\ar[r]\ar[d]& A^1\ar[r]\ar[d]& A^2\ar[r]\ar[d]&\cdots\ar[r]&A^{d-1}\ar[r]\ar[d]& A^{d}\ar[r]\ar[d]& A^{d+1}\ar[r]\ar[d]&0\\
    0\ar[r]&B^0\ar[r]& B^1\ar[r]& B^2\ar[r]&\cdots\ar[r]&B^{d-1}\ar[r]& B^{d}\ar[r]& B^{d+1}\ar[r]&0,
    }
    \end{align*}
    in which each row is a $d$-exact sequence.
\end{enumerate}
\end{defn}

\begin{defn}[{\cite[Definition 3.1]{JG}}]
A \textit{$d$-abelian category} is an additive category $\mathcal{A}$ which satisfies the following axioms:
\begin{enumerate}
    \item[(A0)] The category $\mathcal{A}$ has split idempotents.
    \item[(A1)] Each morphism in $\mathcal{A}$ has a $d$-kernel and a $d$-cokernel.
    \item[(A2)] If $\alpha^0:\xymatrix{A^0\ar[r]&A^1}$ is a monomorphism and
    $\xymatrix{
    A^1\ar[r]&A^2\ar[r]&\cdots\ar[r]& A^{d+1}
    }$
    is a $d$-cokernel of $\alpha^0$, then
    \begin{align*}
        \xymatrix{
    0\ar[r]&A^0\ar[r]^{\alpha^0}& A^1\ar[r]& A^2\ar[r]&\cdots\ar[r]&A^{d-1}\ar[r]&A^d\ar[r]& A^{d+1}\ar[r]&0
    }
    \end{align*}
    is a $d$-exact sequence.
    \item[(A2$^{\text{op}}$)] If $\alpha^d:\xymatrix{A^d\ar[r]&A^{d+1}}$ is an epimorphism and
    $\xymatrix{
    A^0\ar[r]&\cdots\ar[r]& A^{d-1}\ar[r]& A^{d}
    }$
    is a $d$-kernel of $\alpha^d$, then
    \begin{align*}
        \xymatrix{
    0\ar[r]&A^0\ar[r]& A^1\ar[r]& A^2\ar[r]&\cdots\ar[r]&A^{d-1}\ar[r]&A^d\ar[r]^{\alpha^d}& A^{d+1}\ar[r]&0
    }
    \end{align*}
    is a $d$-exact sequence.
\end{enumerate}
\end{defn}

We recall the definition of right minimal morphism, see for example \cite[Definition 1.1, Chapter IV]{A}.
We also recall the definitions of precovers, covers, precovering subcategories and their dual notions, see for example \cite[Definition\ 1.4]{JP}.

\begin{defn} A morphism $\alpha:A\rightarrow B$ in $\mmod\Phi$ is \textit{right minimal} if each morphism $\varphi: A\rightarrow A$ which satifies $\alpha\varphi=\alpha$ is an isomorphism. 
\end{defn}

\begin{defn}\label{defn_cover_ffinite}
Let $\mathcal{X}\subseteq\mathcal{F}\subseteq \mmod\Phi$ be full subcategories. An $\mathcal{X}$\textit{-precover} (or \textit{right $\mathcal{X}$-approximation}) of $A\in\mathcal{F}$ is a morphism of the form $\xi :X\rightarrow A$ with $X\in \mathcal{X}$ such that every morphism $\xi':X'\rightarrow A$ with $X'\in \mathcal{X}$ factorizes as:
\begin{align*}
\xymatrix{
X'\ar[rr]^{\xi'} \ar@{-->}[dr]_{\exists}& & A.\\
& X \ar[ru]_{\xi} &
}
\end{align*}
An $\mathcal{X}$\textit{-cover} (or \textit{minimal right $\mathcal{X}$-approximation}) of $A$ is an $\mathcal{X}$-precover of $A$ which is also a right minimal morphism.
The dual notions of precovers and covers are \textit{preenvelopes} (or \textit{left $\mathcal{X}$-approximations}) and \textit{envelopes} (or \textit{minimal left $\mathcal{X}$-approximation}) respectively.

The subcategory $\mathcal{X}$ of $\mathcal{F}$ is called \textit{precovering} (or \textit{contravariantly finite}) if every object in $\mathcal{F}$ has an $\mathcal{X}$-precover. Dually, it is called \textit{preenveloping} (or \textit{covariantly finite}) if every object in $\mathcal{F}$ has an $\mathcal{X}$-preenvelope. If $\mathcal{X}$ is both precovering and preenveloping, it is called \textit{functorially finite} in $\mathcal{F}$.
\end{defn}

\begin{defn}[{\cite[Definition\ 2.2]{IO}}]\label{defn_dct_sub}
Let $\mathcal{F}$ be a full subcategory  of $\mmod\Phi$. We say that $\mathcal{F}$ is a \textit{$d$-cluster tilting subcategory of $\mmod\Phi$} if:
\begin{enumerate}[label=(\alph*)]
    \item $\mathcal{F}=\{ A\in\mmod\Phi\mid \Ext_{\Phi}^{1\,\dots\, d-1} (\mathcal{F},A)=0 \}=\{ A\in\mathcal{A}\mid \Ext_{\Phi}^{1\,\dots\, d-1} (A,\mathcal{F})=0 \}$,
    \item $\mathcal{F}$ is functorially finite in $\mmod\Phi$.
\end{enumerate}
Note that, by \cite[Theorem 3.16]{JG}, such an $\mathcal{F}$ is a $d$-abelian category. Moreover, a $d$-exact sequence in $\mathcal{F}$ is exact in $\mmod\Phi$.
\end{defn}

In the following sections, we will be studying additive subcategories of $\mathcal{F}$ closed under $d$-extensions.

\begin{defn}\label{defn_additive_subcat}
Let $\mathcal{A}$ be an additive category. An \textit{additive subcategory of $\mathcal{A}$} is a full subcategory which is closed under direct sums, direct summands and isomorphisms in $\mathcal{A}$.
\end{defn}

We introduce Yoneda equivalence in order to define what we mean by an additive subcategory closed under $d$-extensions, see \cite[Chapter IV.9]{HS}.

\begin{defn}
Consider two exact sequences in $\mmod\Phi$ with the same end terms:

\begin{align*}
\xymatrix@R=1em{\epsilon:0\ar[r]&B\ar[r]& C^1\ar[r]& C^2\ar[r]&\cdots\ar[r]&C^{d-1}\ar[r]& C^{d}\ar[r]& A\ar[r]&0,\\
\epsilon':0\ar[r]&B\ar[r]& D^1\ar[r]& D^2\ar[r]&\cdots\ar[r]&D^{d-1}\ar[r]& D^{d}\ar[r]& A\ar[r]&0.
}
\end{align*}
We say that $\epsilon$ and $\epsilon'$ satisfy the relation $\xymatrix{\epsilon\ar@{~>}[r]& \epsilon'}$ if there exists a commutative diagram of the form:
\begin{align*}
\xymatrix@C=2em{\epsilon:\ar[d]&0\ar[r]&B\ar[r]\ar@{=}[d]& C^1\ar[r]\ar[d]& C^2\ar[r]\ar[d]&\cdots\ar[r]&C^{d-1}\ar[r]\ar[d]& C^{d}\ar[r]\ar[d]& A\ar[r]\ar@{=}[d]&0\\
\epsilon':&0\ar[r]&B\ar[r]& D^1\ar[r]& D^2\ar[r]&\cdots\ar[r]&D^{d-1}\ar[r]& D^{d}\ar[r]& A\ar[r]&0.
}
\end{align*}
We say that $\epsilon$ and $\epsilon'$ are \textit{Yoneda equivalent}, and write $\epsilon\sim\epsilon'$, if there exists a chain of exact sequences of the above form $\epsilon=\epsilon_0,\,\epsilon_1,\dots,\,\epsilon_t=\epsilon'$ with
\begin{align*}
    \xymatrix{
    \epsilon_0\ar@{~>}[r]& \epsilon_1& \epsilon_2\ar@{~>}[l]\ar@{~>}[r]&\cdots&\epsilon_t. \ar@{~>}[l]
    }
\end{align*}
We denote the equivalence class of $\epsilon$ by $[\epsilon]$ and the set of all equivalence classes of exact sequences of the above form by Yext$^d_\Phi (A,B)$.
\end{defn}
\begin{remark}\label{remark_Yext_Ext}
Note that Yext$_{\Phi}^d(A,B)$ has a group structure, see \cite[Chapter\ IV.9]{HS}. Moreover, by \cite[Theorem 9.1, Chapter IV.9]{HS}, there is a natural equivalence of set-valued bifunctors Yext$_\Phi^d(-,-)\cong\Ext_{\Phi}^d(-,-)$. Let $\mathcal{F}\subseteq \mmod\Phi$ be $d$-cluster tilting. By \cite[Appendix A]{IO}, if $A,\,B\in\mathcal{F}$, then each equivalence class in Yext$_{\Phi}^d(A,B)$ contains a $d$-exact sequence in $\mathcal{F}$ of the form:
\begin{align*}
    \xymatrix{0\ar[r]&B\ar[r]& F^1\ar[r]^{\varphi^1}& F^2\ar[r]^{\varphi^2}&\cdots\ar[r]^{\varphi^{d-2}}& F^{d-1}\ar[r]^{\varphi^{d-1}}& F^d\ar[r]& A\ar[r]&0,
}
\end{align*}
with $\varphi^1,\dots,\,\varphi^{d-1}$ in $\rad_{\mathcal{F}}$ which is unique up to isomorphism.
So, from now on, we will talk about equivalence classes of $d$-exact sequences in $\Ext_{\Phi}^d$-groups.
\end{remark}

\begin{defn}\label{defn_closed_u_ext}
Let $\mathcal{F}\subseteq \mmod\Phi$ be $d$-cluster tilting. We say that an additive subcategory $\mathcal{X}\subseteq\mathcal{F}$ is \textit{closed under $d$-extensions} if each $d$-exact sequence in $\mathcal{F}$ of the form:
\begin{align*}
\xymatrix{0\ar[r]&X^0\ar[r]& A^1\ar[r]& A^2\ar[r]&\cdots\ar[r]& A^{d-1}\ar[r]& A^d\ar[r]& X^{d+1}\ar[r]&0,
}
\end{align*}
with $X^0$, $X^{d+1}$ in $\mathcal{X}$ is Yoneda equivalent to a $d$-exact sequence in $\mathcal{F}$,
\begin{align*}
\xymatrix{0\ar[r]&X^0\ar[r]& X^1\ar[r]& X^2\ar[r]&\cdots\ar[r]& X^{d-1}\ar[r]& X^d\ar[r]& X^{d+1}\ar[r]&0,
}
\end{align*}
with all terms in $\mathcal{X}$.
\end{defn}

\section{$d$-exact sequences in $\mathcal{F}$ and morphisms between them}\label{section_3}
In this section, working in the following setup, we present some properties of $d$-exact sequences that we will be using in later sections.

\begin{setup}\label{setup}
Let $d$ be a fixed positive integer, $k$ a field, $\Phi$ a finite dimensional $k$-algebra and $\mathcal{F}\subseteq \mmod \Phi$ a $d$-cluster tilting subcategory. Then $\mathcal{F}$ is $d$-abelian.
\end{setup}

\begin{defn}[{\cite[Definition 2.11]{JG}}]\label{defn_dpush}
Consider a complex in $\mathcal{F}$ of the form 
\begin{align*}
\xymatrix{A: &A^0\ar[r]^{\alpha^0}& A^1\ar[r]^{\alpha^1}& A^2\ar[r]&\cdots\ar[r]& A^{d-1}\ar[r]^-{\alpha^{d-1}}& A^d
}
\end{align*}
and a morphism $f^0:A^0\rightarrow B^0$ in $\mathcal{F}$. A \textit{$d$-pushout diagram of $A$ along $f^0$} is a chain map
\begin{align}\label{diagram_d-pushout}
\begin{gathered}
\xymatrix{A\ar[d]^-{\varphi}: &A^0\ar[r]^{\alpha^0}\ar[d]^-{f^0}& A^1\ar[r]^{\alpha^1}\ar[d]^-{f^1}& A^2\ar[r]\ar[d]^-{f^2}&\cdots\ar[r]& A^{d-1}\ar[r]^-{\alpha^{d-1}}\ar[d]^-{f^{d-1}}& A^d\ar[d]^-{f^d}\\
B: &B^0\ar[r]_{\beta^0}& B^1\ar[r]_{\beta^1}& B^2\ar[r]&\cdots\ar[r]& B^{d-1}\ar[r]_-{\beta^{d-1}}& B^d
}
\end{gathered}
\end{align}
 with $B^1,\dots,\, B^{d}$ in $\mathcal{F}$ such that in the mapping cone
\begin{align*}
\xymatrix{C(\varphi): &A^0\ar[r]^-{\gamma^{-1}}& A^1\oplus B^0\ar[r]^{\gamma^0}& A^2\oplus B^1\ar[r]&\cdots\ar[r]& A^{d}\oplus B^{d-1}\ar[r]^-{\gamma^{d-1}}& B^d,
}
\end{align*}
the sequence $(\gamma^0,\dots,\,\gamma^{d-1})$ is a $d$-cokernel of $\gamma^{-1}$. The concept of \textit{$d$-pullback diagram} is defined in a dual way.
\end{defn}

\begin{remark}
By \cite[Theorem 3.8]{JG}, for a complex in $\mathcal{F}$ of the form:
\begin{align*}
\xymatrix{A: &A^0\ar[r]^{\alpha^0}& A^1\ar[r]^{\alpha^1}& A^2\ar[r]&\cdots\ar[r]& A^{d-1}\ar[r]^-{\alpha^{d-1}}& A^d
}
\end{align*}
and a morphism $f^0:A^0\rightarrow B^0$ in $\mathcal{F}$, there is always a $d$-pushout diagram of $A$ along $f^0$ of the form (\ref{diagram_d-pushout}). Moreover, if $\alpha^0$ is a monomorphism, then $\beta^0$ is a monomorphism.
\end{remark}

We can use $d$-pushouts to construct morphisms of $d$-exact sequences in $\mathcal{F}$. The next lemma follows from the dual of \cite[Proposition\ 2.12]{JK}.
\begin{lemma}\label{lemma_d-pushout_exists}
Consider a $d$-exact sequence in $\mathcal{F}$ of the form
\begin{align*}
\xymatrix{\delta: 0\ar[r]&A^0\ar[r]^{\alpha^0}& A^1\ar[r]^{\alpha^1}& A^2\ar[r]&\cdots\ar[r]& A^{d-1}\ar[r]^-{\alpha^{d-1}}& A^d\ar[r]^{\alpha^d}& A^{d+1}\ar[r]&0
}
\end{align*}
and a morphism $f^0:A^0\rightarrow B^0$ in $\mathcal{F}$.
Then there is a $d$-pushout diagram of 
\begin{align*}
\xymatrix{A^0\ar[r]^{\alpha^0}&\cdots\ar[r]^{\alpha^{d-1}}&A^d}
\end{align*}
along $f^0$ and it induces a morphism of $d$-exact sequences of the form:
\begin{align}\label{diagram_d-exactpushout}
\begin{gathered}
\xymatrix @C=2em{\delta\ar[d]^-{f}:&0\ar[r] &A^0\ar[r]^{\alpha^0}\ar[d]^-{f^0}& A^1\ar[r]^{\alpha^1}\ar[d]^-{f^1}& A^2\ar[r]\ar[d]^-{f^2}&\cdots\ar[r]& A^{d-1}\ar[r]^-{\alpha^{d-1}}\ar[d]^-{f^{d-1}}& A^d\ar[d]^-{f^d}\ar[r]^{\alpha^d}& A^{d+1}\ar@{=}[d]\ar[r]&0\\
\epsilon: & 0\ar[r]&B^0\ar[r]_{\beta^0}& B^1\ar[r]_{\beta^1}& B^2\ar[r]&\cdots\ar[r]& B^{d-1}\ar[r]_-{\beta^{d-1}}& B^d\ar[r]_{\beta^d}& A^{d+1}\ar[r]&0.
}
\end{gathered}
\end{align}
\end{lemma}

\begin{notation}
For $A,\,B$ in $\mathcal{F}$, we use the notation $(A,B):=\Hom_{\mathcal{F}}(A,B)$.
\end{notation}

\begin{lemma}\label{lemma_null-homotopic}
Consider a morphism $h$ of $d$-exact sequences in $\mathcal{F}$ of the form:
\begin{align*}
\xymatrix@C=2em{\delta\ar[d]^-{h}:&0\ar[r] &A^0\ar[r]^{\alpha^0}\ar[d]^>>>>>>{h^0}& A^1\ar[r]^{\alpha^1}\ar[d]^>>>>>>{h^1}\ar@{-->}[ld]^{s^1}& A^2\ar[r]\ar[d]^>>>>>>{h^2}\ar@{-->}[ld]^{s^2}&\cdots\ar[r]& A^{d-1}\ar[r]^-{\alpha^{d-1}}\ar[d]^>>>>>>{h^{d-1}}& A^d\ar[d]^>>>>>>{h^d}\ar[r]^{\alpha^d}\ar@{-->}[ld]^{s^d}& A^{d+1}\ar[d]^>>>>>>{h^{d+1}}\ar[r]\ar@{-->}[ld]^{s^{d+1}}&0\\
\epsilon: & 0\ar[r]&B^0\ar[r]_{\beta^0}& B^1\ar[r]_{\beta^1}& B^2\ar[r]&\cdots\ar[r]& B^{d-1}\ar[r]_-{\beta^{d-1}}& B^d\ar[r]_{\beta^d}& B^{d+1}\ar[r]&0.
}
\end{align*}
The following are equivalent:
\begin{enumerate}[label=(\alph*)]
    \item there is a morphism $s^{d+1}:A^{d+1}\rightarrow B^d$ such that $\beta^d s^{d+1}=h^{d+1},$
    \item there is a morphism $s^{1}:A^{1}\rightarrow B^0$ such that $s^{1} \alpha^0=h^{0},$
    \item the morphism $h:\delta\rightarrow\epsilon$ is null-homotopic.
\end{enumerate}
\end{lemma}

\begin{proof}
It is clear that (c) implies both (a) and (b). Assume (a) holds. By the definition of $d$-kernel, applying $(A^d,-)$ to $\epsilon$, we obtain the exact sequence:
\begin{align*}
\xymatrix{
    (A^d,B^{d-1})\ar[r]^-{\beta^{d-1}_{*}}&(A^d,B^d)\ar[r]^-{\beta^{d}_{*}}& (A^d, B^{d+1}).
    }
\end{align*}
Note that
\begin{align*}
\beta^{d}_{*} (h^d-s^{d+1}\alpha^d)=\beta^d h^d-\beta^d s^{d+1}\alpha^d=\beta^d h^d-h^{d+1}\alpha^d=0,
\end{align*}
so that $h^d-s^{d+1}\alpha^d$ is in $\ker \beta^{d}_{*}=\im \beta^{d-1}_{*}$. So there exists a morphism $s^d:A^d\rightarrow B^{d-1}$ such that $\beta^{d-1}s^d=h^d-s^{d+1}\alpha^d$. Inductively, for $i=d-1,\, d-2,\dots,\,1$, we obtain $s^i:A^i\rightarrow B^{i-1}$ such that $h^i=\beta^{i-1} s^i+s^{i+1}\alpha^i$. Then,
\begin{align*}
    \beta^0 s^1 \alpha^0=h^1\alpha^0-s^2\alpha^1\alpha^0=h^1\alpha^0=\beta^0 h^0.
\end{align*}
Since $\beta^0$ is a monomorphism, it follows that $s^1 \alpha^0=h^0$. So (b) and (c) hold. Dually, (b) implies both (a) and (c).
\end{proof}

The special case when $\delta=\epsilon$ and $h$ is the identity on $\delta$ in Lemma \ref{lemma_null-homotopic} gives the following. 
\begin{corollary}\label{coro_split}
Consider a $d$-exact sequence in $\mathcal{F}$ of the form
\begin{align*}
\xymatrix{\delta: 0\ar[r]&A^0\ar[r]^{\alpha^0}& A^1\ar[r]^{\alpha^1}& A^2\ar[r]&\cdots\ar[r]& A^{d-1}\ar[r]^-{\alpha^{d-1}}& A^d\ar[r]^{\alpha^d}& A^{d+1}\ar[r]&0.
}
\end{align*}
The following are equivalent:
\begin{enumerate}[label=(\alph*)]
    \item $\alpha^0$ is a split monomorphism,
    \item $\alpha^d$ is a split epimorphism,
    \item the identity on $\delta$ is null-homotopic.
\end{enumerate}
If any, and so all, of the above hold, we say that $\delta$ is a \textit{split $d$-exact sequence}.
\end{corollary}

\begin{remark}\label{remark_Ext_pull_push}
By Remark \ref{remark_Yext_Ext}, if $A^0,\,A^{d+1}\in\mathcal{F}$, then every element in $\Ext_{\Phi}^d(A^{d+1},A^0)$ is given by a $d$-exact sequence in $\mathcal{F}$. Consider a $d$-exact sequence in $\mathcal{F}$ of the form 
\begin{align*}
\xymatrix{\delta: 0\ar[r]&A^0\ar[r]^{\alpha^0}& A^1\ar[r]^{\alpha^1}& A^2\ar[r]&\cdots\ar[r]& A^{d-1}\ar[r]^-{\alpha^{d-1}}& A^d\ar[r]^{\alpha^d}& A^{d+1}\ar[r]&0.
}
\end{align*}
\begin{enumerate}[label=(\alph*)]
\item
By \cite[Lemma 1.6]{JJ}, if $[\delta]=0$ in $\Ext_{\Phi}^d(A^{d+1}, A^0)$, then $\delta$ is a split $d$-exact sequence. Moreover, it can be checked that if $\delta$ is a split $d$-exact sequence, then $[\delta]=0$.
\item Given a morphism $f^0: A^0\rightarrow B^{0}$ in $\mathcal{F}$, we can look at the morphism
\begin{align*}
    \Ext_{\Phi}^d(A^{d+1},f^0):\Ext_{\Phi}^d(A^{d+1}, A^0)\rightarrow \Ext_{\Phi}^d(A^{d+1}, B^0)
\end{align*}
in terms of $d$-exact sequences in $\mathcal{F}$.
For $\delta$ as above, $f^0\cdot\delta:=\Ext_{\Phi}^d(A^{d+1},f^0)(\delta)$ is given by extending a $d$-pushout diagram as in (\ref{diagram_d-exactpushout}) from Lemma \ref{lemma_d-pushout_exists}:
\begin{align*}
\xymatrix@C=2em{\delta:\,0\ar[d]^f\ar[r]& A^0\ar[r]^{\alpha^0}\ar[d]^-{f^0}& A^1\ar[r]^{\alpha^1}\ar[d]^-{f^1}& A^2\ar[r]\ar[d]^-{f^2}&\cdots\ar[r]& A^{d-1}\ar[r]^-{\alpha^{d-1}}\ar[d]^-{f^{d-1}}& A^d\ar[d]^-{f^d}\ar[r]^{\alpha^d}& A^{d+1}\ar@{=}[d]\ar[r]&0\\
f^0\cdot\delta: 0\ar[r]&B^0\ar[r]_{\beta^0}& B^1\ar[r]_{\beta^1}& B^2\ar[r]&\cdots\ar[r]& B^{d-1}\ar[r]_-{\beta^{d-1}}& B^d\ar[r]_{\beta^d}& A^{d+1}\ar[r]&0.
}
\end{align*}
Dually, for $g^{d+1}:B^{d+1}\rightarrow A^{d+1}$ in $\mathcal{F}$, we have that $\delta\cdot g^{d+1}:=\Ext_{\Phi}^d(g^{d+1},A^0)(\delta)\in \Ext_{\Phi}^d(B^{d+1}, A^0)$ is given by a $d$-pullback diagram.
This construction can be seen in the $d=1$ case in \cite[Section III.1 and Theorem III.2.4]{HS}. The case for general $d\geq 1$ follows by methods similar to those used in \cite[Section IV.9]{HS}.
\end{enumerate}
\end{remark}

\begin{lemma}\label{lemma_complete_to_morph}
Suppose there are $d$-exact sequences $\delta$ and $\epsilon$ in $\mathcal{F}$ and, for some $0\leq i<j\leq d$, there are morphims $f^i,\, f^{i+1},\dots,\,  f^j$ such that $\beta^lf^l=f^{l+1}\alpha^l$ for $i\leq l\leq j-1$, \textbf{i.e.} the following diagram commutes:
\begin{align*}
\xymatrix@C=1.6em{\delta:0\ar[r] &A^0\ar[r]^{\alpha^0}\ar@{-->}[d]^-{f^0}&\cdots\ar[r]& A^{i-1}\ar[r]^-{\alpha^{i-1}}\ar@{-->}[d]^{f^{i-1}}& A^{i}\ar[r]^{\alpha^i}\ar[d]^{f^i}&  \cdots \ar[r]^{\alpha^{j-1}}& A^{j}\ar[r]^{\alpha^j}\ar[d]^{f^{j}}& A^{j+1}\ar[r]^{\alpha^{j+1}}\ar@{-->}[d]^{f^{j+1}}&\cdots\ar[r]^{\alpha^{d+1}}&  A^{d+1}\ar@{-->}[d]^{f^{d+1}}\ar[r]&0\\
\epsilon:0\ar[r]&B^0\ar[r]_{\beta^0}&\cdots\ar[r]& B^{i-1}\ar[r]_-{\beta^{i-1}} & B^i\ar[r]_{\beta^i}& \cdots \ar[r]_{\beta^{j-1}}&B^j\ar[r]_{\beta^j} & B^{j+1}\ar[r]_{\beta^{j+1}}& \cdots\ar[r]_{\beta^{d+1}} & B^{d+1}\ar[r]&0.
}
\end{align*}
Then, for $0\leq l\leq i-1$ and $j+1\leq l\leq d+1$, there exist morphisms $f^l:A^l\rightarrow B^l$ completing $f^i,\, f^{i+1},\dots,\,  f^j$ to a morphism of $d$-exact sequences.
\end{lemma}

\begin{proof}
To construct the morphisms $f^l$ for $0\leq l\leq i-1$, use the fact that
\begin{align*}
   \xymatrix{
   0\ar[r]& B^0\ar[r]^{\beta^0}& B^1\ar[r]&\cdots\ar[r]^{\beta^{d-1}}& B^d
   }
\end{align*}
is a $d$-kernel of $\beta^d:B^d\rightarrow B^{d+1}$. To construct $f^l$ for $j+1\leq l\leq d+1$, use the fact that
\begin{align*}
   \xymatrix{
   A^1\ar[r]^{\alpha^1}& A^2\ar[r]&\cdots\ar[r]^{\alpha^{d}}& A^{d+1}\ar[r]&0
   }
\end{align*}
is a $d$-cokernel of $\alpha^0:A^0\rightarrow A^1$.
\end{proof}

We recall the definition of Jacobson radical of $\mathcal{F}$, see for example \cite[Definition A.3.3]{A}.

\begin{defn}
The \textit{Jacobson radical of $\mathcal{F}$} is the two sided ideal $\rad_\mathcal{F}$ in $\mathcal{F}$ defined by the formula
\begin{align*}
    \rad_\mathcal{F}(A,B)=\{ \alpha: A\rightarrow B\mid 1_A-\beta\alpha \text{ is invertible for any } \beta:B\rightarrow A \},
\end{align*}
for all objects $A$ and $B$ in $\mathcal{F}$.
\end{defn}

The following lemma can be deduced from \cite[Lemma 1.1]{JK}.
\begin{lemma}\label{lemma_rad_rmin}
Consider a $d$-exact sequence in $\mathcal{F}$ of the form
\begin{align*}
\xymatrix{\delta: 0\ar[r]&A^0\ar[r]^{\alpha^0}& A^1\ar[r]^{\alpha^1}& A^2\ar[r]&\cdots\ar[r]& A^{d-1}\ar[r]^-{\alpha^{d-1}}& A^d\ar[r]^{\alpha^d}& A^{d+1}\ar[r]&0.
}
\end{align*}
For $i=1,\dots,\, d$, we have that $\alpha^i$ is right minimal if and only if $\alpha^{i-1}$ is in $\rad_{\mathcal{F}}$.
\end{lemma}

\begin{lemma}\label{lemma_iso_d-exact_seq}
Consider a $d$-exact sequence in $\mathcal{F}$ of the form
\begin{align*}
\xymatrix{\delta: 0\ar[r]&A^0\ar[r]^{\alpha^0}& A^1\ar[r]^{\alpha^1}& A^2\ar[r]&\cdots\ar[r]& A^{d-1}\ar[r]^-{\alpha^{d-1}}& A^d\ar[r]^{\alpha^d}& A^{d+1}\ar[r]&0,
}
\end{align*}
with $\alpha^0,\dots,\, \alpha^{d-1}$ in $\rad_\mathcal{F}$ and a morphism of $d$-exact sequences:
\begin{align*}
\xymatrix@C=2em{\delta:\ar[d]^f &0\ar[r]&A^0\ar[r]^{\alpha^0}\ar[d]^{f^0}& A^1\ar[r]^{\alpha^1}\ar[d]^{f^1}& A^2\ar[r]\ar[d]^{f^2}&\cdots\ar[r]& A^{d-1}\ar[r]^-{\alpha^{d-1}}\ar[d]^{f^{d-1}}& A^d\ar[r]^{\alpha^d}\ar[d]^{f^d}& A^{d+1}\ar[r]\ar@{=}[d]&0\\
\delta: &0\ar[r]&A^0\ar[r]_{\alpha^0}& A^1\ar[r]_{\alpha^1}& A^2\ar[r]&\cdots\ar[r]& A^{d-1}\ar[r]_-{\alpha^{d-1}}& A^d\ar[r]_{\alpha^d}& A^{d+1}\ar[r]&0,
}
\end{align*}
where $f^d$ is an isomorphism. Then $f^0,\dots,\, f^{d-1}$ are all isomorphisms.
\end{lemma}

\begin{proof}
First note that, by Lemma \ref{lemma_rad_rmin}, since $\alpha^0,\dots,\, \alpha^{d-1}$ are in $\rad_\mathcal{F}$ then $\alpha^1$, $\dots$, $\alpha^d$ are right minimal. Since $f^d$ is invertible,  $\alpha^d f^d=\alpha^d$ implies that $\alpha^d=\alpha^d (f^d)^{-1}$. Then, using Lemma \ref{lemma_complete_to_morph}, we can construct a commutative diagram of the form:
\begin{align*}
\xymatrix@C=2em{\delta:\ar[d]^f &0\ar[r]&A^0\ar[r]^{\alpha^0}\ar[d]^{f^0}& A^1\ar[r]^{\alpha^1}\ar[d]^{f^1}& A^2\ar[r]\ar[d]^{f^2}&\cdots\ar[r]& A^{d-1}\ar[r]^-{\alpha^{d-1}}\ar[d]^{f^{d-1}}& A^d\ar[r]^{\alpha^d}\ar[d]^{f^d}& A^{d+1}\ar[r]\ar@{=}[d]&0\\
\delta:\ar[d]^g &0\ar[r]&A^0\ar[r]^{\alpha^0}\ar[d]^{g^0}& A^1\ar[r]^{\alpha^1}\ar[d]^{g^1}& A^2\ar[r]\ar[d]^{g^2}&\cdots\ar[r]& A^{d-1}\ar[r]^-{\alpha^{d-1}}\ar[d]^{g^{d-1}}& A^d\ar[r]^{\alpha^d}\ar[d]^{(f^d)^{-1}}& A^{d+1}\ar[r]\ar@{=}[d]&0\\
\delta: &0\ar[r]&A^0\ar[r]_{\alpha^0}& A^1\ar[r]_{\alpha^1}& A^2\ar[r]&\cdots\ar[r]& A^{d-1}\ar[r]_-{\alpha^{d-1}}& A^d\ar[r]_{\alpha^d}& A^{d+1}\ar[r]&0.
}
\end{align*}
Hence $\alpha^{d-1}=\alpha^{d-1}g^{d-1}f^{d-1}$ and as $\alpha^{d-1}$ is right minimal, it follows that $g^{d-1} f^{d-1}$ is an isomorphism. Similarly, looking at $fg$ we conclude that $f^{d-1}g^{d-1}$ is an isomorphism and hence $f^{d-1}$ is an isomorphism.
Letting $h^{d-1}:=(g^{d-1}f^{d-1})^{-1}$, we can construct a commutative diagram of the form:
\begin{align*}
\xymatrix@C=2em{\delta:\ar[d]^{gf} &0\ar[r]&A^0\ar[r]^{\alpha^0}\ar[d]^{g^0f^0}& A^1\ar[r]^{\alpha^1}\ar[d]^{g^1f^1}&\cdots\ar[r]& A^{d-2}\ar[r]^{\alpha^{d-2}}\ar[d]^{g^{d-2}f^{d-2}} &A^{d-1}\ar[r]^-{\alpha^{d-1}}\ar[d]^{g^{d-1}f^{d-1}}& A^d\ar[r]^{\alpha^d}\ar@{=}[d]& A^{d+1}\ar[r]\ar@{=}[d]&0\\
\delta:\ar[d]^h &0\ar[r]&A^0\ar[r]^{\alpha^0}\ar@{-->}[d]^{h^0}& A^1\ar[r]^{\alpha^1}\ar@{-->}[d]^{h^1}&\cdots\ar[r]& A^{d-2}\ar[r]^{\alpha^{d-2}}\ar@{-->}[d]^{h^{d-2}} &A^{d-1}\ar[r]^-{\alpha^{d-1}}\ar[d]^{h^{d-1}}& A^d\ar[r]^{\alpha^d}\ar@{=}[d]& A^{d+1}\ar[r]\ar@{=}[d]&0\\
\delta: &0\ar[r]&A^0\ar[r]_{\alpha^0}& A^1\ar[r]_{\alpha^1}&\cdots\ar[r]&A^{d-2}\ar[r]_{\alpha^{d-2}}& A^{d-1}\ar[r]_-{\alpha^{d-1}}& A^d\ar[r]_{\alpha^d}& A^{d+1}\ar[r]&0.
}
\end{align*}
Then
\begin{align*}
\alpha^{d-2}=h^{d-1}g^{d-1} f^{d-1}\alpha^{d-2}=\alpha^{d-2}h^{d-2}g^{d-2}f^{d-2},
\end{align*}
and, as $\alpha^{d-2}$ is right minimal, we have that $h^{d-2}g^{d-2}f^{d-2}$ is an isomorphism. Similarly, $g^{d-2}f^{d-2}h^{d-2}$ is an isomorphism. Then $g^{d-2}f^{d-2}$ is an isomorphism. Since also $f^{d-1}g^{d-1}$ is an isomorphism, by a similar argument we have that $f^{d-2}g^{d-2}$ is an isomorphism. Hence $f^{d-2}$ is an isomorphism. Proceeding by induction, we conclude that $f^1,\dots,\,f^{d-2}$ are all isomorphisms. Then also $f^0$ is forced to be an isomorphism, because $\alpha^0$ is a monomorphism.
\end{proof}

\begin{lemma}\label{lemma_eq_classes_d-pushout}
Consider a $d$-exact sequence in $\mathcal{F}$ of the form 
\begin{align*}
\xymatrix{\delta: 0\ar[r]&A^0\ar[r]^{\alpha^0}& A^1\ar[r]^{\alpha^1}& A^2\ar[r]&\cdots\ar[r]& A^{d-1}\ar[r]^-{\alpha^{d-1}}& A^d\ar[r]^{\alpha^d}& A^{d+1}\ar[r]&0
}
\end{align*}
and a morphism $f^0:A^0\rightarrow B^0$ in $\mathcal{F}$. Let $f:\delta\rightarrow f^0\cdot \delta$ be as described in Remark \ref{remark_Ext_pull_push}(b). Suppose there is a morphism of $d$-exact sequences of the form:
\begin{align*}
\xymatrix@C=2em{\delta\ar[d]^-{g}:&0\ar[r] &A^0\ar[r]^{\alpha^0}\ar[d]_-{g^0=f^0}& A^1\ar[r]^{\alpha^1}\ar[d]^-{g^1}& A^2\ar[r]\ar[d]^-{g^2}&\cdots\ar[r]& A^{d-1}\ar[r]^-{\alpha^{d-1}}\ar[d]^-{g^{d-1}}& A^d\ar[d]^-{g^d}\ar[r]^{\alpha^d}& A^{d+1}\ar@{=}[d]\ar[r]&0\\
\epsilon': & 0\ar[r]&B^0\ar[r]_{\gamma^0}& C^1\ar[r]_{\gamma^1}& C^2\ar[r]&\cdots\ar[r]&C^{d-1}\ar[r]_-{\gamma^{d-1}}& C^d\ar[r]_{\gamma^d}& A^{d+1}\ar[r]&0.
}
\end{align*}
Then $[f^0\cdot\delta]=[\epsilon']$ in $\Ext_{\Phi}^d(A^{d+1}, B^0)$.
\end{lemma}
\begin{proof}
Note that $f^0\cdot\delta$ as described in Remark \ref{remark_Ext_pull_push}(b) is obtained by extending a $d$-pushout diagram. The result then follows using \cite[Proposition 4.8]{JG}.
\end{proof}

\begin{defn}[{\cite[Appendix A]{IO}}]
When $d\geq 2$, we say that a $d$-exact sequence in $\mathcal{F}$ of the form:
\begin{align*}
\xymatrix{\delta: 0\ar[r]&A^0\ar[r]^{\alpha^0}& A^1\ar[r]^{\alpha^1}& A^2\ar[r]&\cdots\ar[r]& A^{d-1}\ar[r]^-{\alpha^{d-1}}& A^d\ar[r]^{\alpha^d}& A^{d+1}\ar[r]&0
}
\end{align*}
is \textit{almost minimal} if $\alpha^1,\dots,\,\alpha^{d-1}$ are in $\rad_\mathcal{F}$.
\end{defn}

\begin{remark}\label{remark_almost_minimal}
Let $\mathcal{X}\subseteq\mathcal{F}$ be an additive subcategory closed under $d$-extensions.
By \cite[appendix A]{IO}, in every Yoneda equivalence class, there is a unique almost minimal sequence up to isomorphism. Consider a $d$-exact sequence in $\mathcal{F}$ of the form:
\begin{align*}
\xymatrix{\delta: 0\ar[r]&X^0\ar[r]& A^1\ar[r]& A^2\ar[r]&\cdots\ar[r]& A^{d-1}\ar[r]& A^d\ar[r]& X^{d+1}\ar[r]&0,
}
\end{align*}
with $X^0,\,X^{d+1}$ in $\mathcal{X}$. The almost minimal sequence in the equivalence class $[\delta]$ has all the terms in $\mathcal{X}$. In fact, since $\mathcal{X}$ is closed under $d$-extensions, we know there is a $d$-exact sequence with all terms in $\mathcal{X}$ in $[\delta]$, and dropping extra direct summands of the form $X\xrightarrow{\sim} X$ in the middle terms of this, we obtain the unique almost minimal sequence in $[\delta]$, say
\begin{align*}
    \xymatrix@C=2.2em{
    \delta': 0\ar[r]&X^0\ar[r]& X^1\ar[r]& X^2\ar[r]&\cdots\ar[r]& X^{d-1}\ar[r]& X^d\ar[r]& X^{d+1}\ar[r]&0,
    }
\end{align*}
with all terms in $\mathcal{X}$.
Note that dropping extra direct summands of the form $A\xrightarrow{\sim} A$ in the middle terms of $\delta$, we also obtain an almost minimal sequence
\begin{align*}
\xymatrix{\epsilon: 0\ar[r]&X^0\ar[r]& \overline{A^1}\ar[r]& \overline{A^2}\ar[r]&\cdots\ar[r]& \overline{A^{d-1}}\ar[r]& \overline{A^d}\ar[r]& X^{d+1}\ar[r]&0.
}
\end{align*}
By uniqueness, $\delta'\cong\epsilon$ and so $\epsilon$ has all terms in $\mathcal{X}$. Note that $[\delta]=[\epsilon]$ and, since $\overline{A^i}$ is a direct summand of $A^i$ for any $i=1,\dots,\, d$, there are morphisms of $d$-exact sequences $\epsilon\rightarrow\delta$ and $\delta\rightarrow\epsilon$.
\end{remark}

\section{$d$-Auslander-Reiten sequences in $\mathcal{X}$}\label{section_4}
\begin{setup}
Let $d$, $\Phi$ and $\mathcal{F}$ be as in Setup \ref{setup} and let $\mathcal{X}\subseteq\mathcal{F}$ be an additive subcategory closed under $d$-extensions.
\end{setup}
We introduce $d$-Auslander-Reiten sequences in the subcategory $\mathcal{X}$ and give equivalent definitions. Note that the case $\mathcal{X}=\mathcal{F}$ will give the corresponding results in the ambient category $\mathcal{F}$.
\begin{defn}
A morphism $\xi^d:X^d\rightarrow X^{d+1}$ in $\mathcal{X}$ is \textit{right almost split in $\mathcal{X}$} if it is not a split epimorphism and for every $Y$ in $\mathcal{X}$, every morphism $\eta:Y\rightarrow X^{d+1}$ which is not a split epimorphism factors through $\xi^d$, \textbf{i.e.} there exists a morphism $Y\rightarrow X^d$ such that the following diagram commutes:
    \begin{align*}
        \xymatrix  {
        X^d\ar[rr]^{\xi^d}&& X^{d+1}.\\
        & Y\ar[ru]_\eta \ar@{-->}[lu]^{\exists} &
        }
    \end{align*}
Dually, one defines \textit{left almost split morphisms in $\mathcal{X}$}.
\end{defn}

\begin{defn}\label{defn_dAR_seq}
We say that a $d$-exact sequence in $\mathcal{F}$ with all terms from $\mathcal{X}$ of the form
    \begin{align*}
    \xymatrix{\epsilon: 0\ar[r]&X^0\ar[r]^-{\xi^0}& X^1\ar[r]^{\xi^1}&\cdots\ar[r]& X^{d-1}\ar[r]^-{\xi^{d-1}} & X^d\ar[r]^{\xi^d}& X^{d+1}\ar[r]&0,
    }
    \end{align*}
is a \textit{$d$-Auslander-Reiten sequence in $\mathcal{X}$} if the morphism $\xi^0$ is left almost split in $\mathcal{X}$, the morphism $\xi^d$ is right almost split in $\mathcal{X}$ and, when $d\geq 2$, also $\xi^1,\dots,\,\xi^{d-1}\in\rad_\mathcal{X}$.
\end{defn}

The following is a well known result, see \cite[Lemma V.1.7]{ARS}. Note that for a module in $\mmod \Phi$, having local endomorphism ring is equivalent to being indecomposable.
\begin{lemma}\label{lemma_las_rad}
Let $\xi^0:X^0\rightarrow X^1$ be left almost split in $\mathcal{X}$. Then $\End_\Phi(X^0)$ is local and $\xi^0$ is in $\rad_\mathcal{X}$.
\end{lemma}

\begin{remark}
Note that if $\epsilon$ is a $d$-Auslander-Reiten sequence in $\mathcal{X}$, Lemma \ref{lemma_las_rad} and its dual imply that $\End_\Phi(X^0)$ and $\End_\Phi(X^{d+1})$ are local and $\xi^0,\,\xi^d$ are in $\rad_\mathcal{X}$.
\end{remark}

\begin{lemma}\label{lemma_dAR_equiv}
Consider a $d$-exact sequence in $\mathcal{F}$ with all terms from $\mathcal{X}$ of the form:
    \begin{align*}
    \xymatrix{\epsilon: 0\ar[r]&X^0\ar[r]^-{\xi^0}& X^1\ar[r]^{\xi^1}&\cdots\ar[r]& X^{d-1}\ar[r]^-{\xi^{d-1}} & X^d\ar[r]^{\xi^d}& X^{d+1}\ar[r]&0.
    }
    \end{align*}
The following are equivalent:
\begin{enumerate}[label=(\alph*)]
    \item $\epsilon$ is a $d$-Auslander-Reiten sequence in $\mathcal{X}$,
    \item $\xi^0,\,\xi^1,\dots,\,\xi^{d-1}$ are in $\rad_\mathcal{X}$ and $\xi^d$ is right almost split in $\mathcal{X}$,
    \item $\xi^1,\dots,\,\xi^{d-1},\,\xi^d$ are in $\rad_\mathcal{X}$ and $\xi^0$ is left almost split in $\mathcal{X}$.
\end{enumerate}
\end{lemma}

\begin{proof}
By Lemma \ref{lemma_las_rad} and its dual, it is clear that (a) implies both (b) and (c). Suppose now that (b) holds. By the dual of Lemma \ref{lemma_las_rad}, it follows that $\xi^d\in\rad_{\mathcal{X}}$. Let $f^0:X^0\rightarrow Y^0$ be a morphism in $\mathcal{X}$ that is not a split monomorphism. By Lemma \ref{lemma_d-pushout_exists}, there is a morphism of $d$-exact sequences of the form:
 \begin{align*}
    \xymatrix{\epsilon\ar[d]^f:&0\ar[r] &X^0\ar[r]^{\xi^0}\ar[d]^-{f^0}& X^1\ar[r]^{\xi^1}\ar[d]^-{f^1}& \cdots\ar[r]&X^{d-1}\ar[r]^-{\xi^{d-1}}\ar[d]^-{f^{d-1}}&  X^d\ar[d]^-{f^d}\ar[r]^{\xi^d}& X^{d+1}\ar@{=}[d]\ar[r]&0\\
    \delta: & 0\ar[r]&Y^0\ar[r]_{\eta^0}& Y^1\ar[r]_{\eta^1}& \cdots\ar[r]&Y^{d-1}\ar[r]_-{\eta^{d-1}}& Y^d\ar[r]_{\eta^d}& X^{d+1}\ar[r]&0,
    }
\end{align*}
where we may assume $Y^1,\dots,\,Y^d$ are in $\mathcal{X}$ by Remark \ref{remark_almost_minimal}.
Suppose for a contradiction that $\eta^0$ is not a split monomorphism. Then $\eta^d$ is not a split epimorphism by Corollary \ref{coro_split} and, since $\xi^d$ is right almost split in $\mathcal{X}$, then there exists $g^d:Y^d\rightarrow X^d$ such that $\xi^d g^d=\eta^d$. By Lemma \ref{lemma_complete_to_morph}, there is a morphism of $d$-exact sequences of the form:
 \begin{align*}
    \xymatrix{
    \delta:\ar[d]^{g} & 0\ar[r]&Y^0\ar[r]^{\eta^0}\ar[d]^{g^0}& Y^1\ar[r]^{\eta^1}\ar[d]^{g^1}& \cdots\ar[r]&Y^{d-1}\ar[r]^-{\eta^{d-1}}\ar[d]^{g^{d-1}}& Y^d\ar[r]^{\eta^d}\ar[d]^{g^d}& X^{d+1}\ar[r]\ar@{=}[d]&0\\
    \epsilon:&0\ar[r] &X^0\ar[r]_{\xi^0}& X^1\ar[r]_{\xi^1}& \cdots\ar[r]&X^{d-1}\ar[r]_-{\xi^{d-1}}&  X^d\ar[r]_{\xi^d}& X^{d+1}\ar[r]&0.
    }
\end{align*}
Note that $\xi^dg^df^d=\xi^d$ and, since Lemma \ref{lemma_rad_rmin} implies that $\xi^d$ is right minimal, it follows that $g^df^d$ is an isomorphism. Hence, Lemma \ref{lemma_iso_d-exact_seq} implies that $g^0f^0$ is an isomorphism. so that $f^0$ is a split monomorphism, contradicting our assumption. So $\eta^0$ is a split monomorphism and there is a morphism $\mu:Y^1\rightarrow Y^0$ such that $\mu\eta^0=1_{Y^0}$. Then
\begin{align*}
    \mu f^1 \xi^0=\mu\eta^0 f^0=f^0,
\end{align*}
so $\xi^0$ is left almost split in $\mathcal{X}$ and we have proved (c). Dually, (c) implies (b) and it is clear that both (b) and (c) imply (a).
\end{proof}

\section{$\mathcal{X}$-covers and the left end term of a $d$-Auslander-Reiten sequence in $\mathcal{X}$}\label{section_Kleiner2}

In this section, we generalise the results in \cite[Section 2]{MK} on $\mmod\Phi$ to its higher analogue $\mathcal{F}$.
Iyama proved in \cite[Theorem 3.3.1]{IO} that if $A^{d+1}\in\mathcal{F}$ is an indecomposable non-projective, then there exists a $d$-Auslander-Reiten sequence in $\mathcal{F}$ ending at $A^{d+1}$ and starting at $D\Tr_d (A^{d+1})$, see Proposition \ref{thm_dAR_F_endt}.
The idea is to give an analogue of this result for $d$-Auslander-Reiten sequences in $\mathcal{X}$.
Consider an indecomposable $X$ in $\mathcal{X}$ that admits non-split $d$-exact sequences ending at it with terms in $\mathcal{X}$. We ``approximate'' $D\Tr_d(X)$ with an indecomposable $\sigma X$ in $\mathcal{X}$. We show there is a $d$-Auslander-Reiten sequence in $\mathcal{X}$ ending in $X$ and that this sequence is forced to start in $\sigma X$.

Recall the definition of $\mathcal{X}$-cover from Definition \ref{defn_cover_ffinite}. 
Note that the duals of all the results presented in this section are also true.

\begin{lemma}\label{lemma_cover_mono}
Let $A\in\mathcal{F}$ and $g:X\rightarrow A$ be an $\mathcal{X}$-cover. Then,
\begin{align*}
    \Ext_{\Phi}^d (-,g)\mid_{\mathcal{X}}: \Ext_{\Phi}^d (-,X)\mid_{\mathcal{X}}\,\longrightarrow \Ext_{\Phi}^d (-,A)\mid_{\mathcal{X}}
\end{align*}
is a monomorphism of contravariant functors.
\end{lemma}

\begin{proof}
Given a $d$-exact sequence in $\mathcal{F}$ of the form:

\begin{align*}
\xymatrix{\delta: 0\ar[r]&X\ar[r]^{\xi^0}& X^1\ar[r]^{\xi^1}& X^2\ar[r]&\cdots\ar[r]& X^{d-1}\ar[r]^-{\xi^{d-1}}& X^d\ar[r]^{\xi^d}& X^{d+1}\ar[r]&0,
}
\end{align*}
where $X^{d+1}$ is in $\mathcal{X}$. Since $\mathcal{X}$ is closed under $d$-extensions, we may assume that $X^1,\dots,X^{d}$ are also in $\mathcal{X}$. Consider the morphism of $d$-exact sequences in $\mathcal{F}$ obtained as in Remark \ref{remark_Ext_pull_push}(b):
\begin{align*}
\xymatrix{\delta\ar[d]:&0\ar[r] &X\ar[r]^{\xi^0}\ar[d]^-{g}& X^1\ar[r]^{\xi^1}\ar[d]^-{g^1}& \cdots\ar[r]&  X^d\ar[d]^-{g^d}\ar[r]^{\xi^d}& X^{d+1}\ar@{=}[d]\ar[r]&0\\
 g\cdot \delta: & 0\ar[r]&A\ar[r]_{\alpha^0}& A^1\ar[r]_{\alpha^1}& \cdots\ar[r]& A^d\ar[r]_{\alpha^d}& X^{d+1}\ar[r]&0.
}
\end{align*}
Suppose that $g\cdot\delta$ splits, \textbf{i.e.} $[g\cdot\delta]=0$. By Remark \ref{remark_Ext_pull_push}(a), we want to prove that also $\delta$ splits so that $\Ext_{\Phi}^d (-,g)\mid_{\mathcal{X}}$ is a monomorphism.
By Remark \ref{remark_Ext_pull_push}(a), there exists a morphism $\gamma: A^1\rightarrow A$ such that $\gamma\alpha^0=1_A$. Then
\begin{align*}
    g=\gamma\alpha^0 g=\gamma g^1\xi^0.
\end{align*}
Moreover, since $X^1$ is in $\mathcal{X}$ and $g$ is an $\mathcal{X}$-cover, there is a morphism $\eta:X^1\rightarrow X$ such that $g\eta=\gamma g^1$. Then, we have
\begin{align*}
    g=\gamma g^1 \xi^0=g \eta \xi^0.
\end{align*}
As $g$ is right minimal, it follows that $\eta \xi^0$ is an isomorphism. This implies that $\xi^0$ is a split monomorphism and so $\delta$ splits, \textbf{i.e.} $[\delta]=0$ in $\Ext_{\Phi}^d (X^{d+1}, X)$.
\end{proof}

In \cite[Theorem 3.3.1]{IO}, Iyama shows that the end terms of a $d$-Auslander-Reiten sequence in $\mathcal{F}$ determine each other. We recall this result focusing on the right end term of $d$-Auslander-Reiten sequences.

\begin{defn}
[{\cite[1.4.1]{IO}}]\label{defn_d_higher_transpose}
Let $M\in\mmod\Phi$ and consider an augmented projective resolution of $M$ of the form:
\begin{align*}
    \cdots\rightarrow P_2\rightarrow P_1\rightarrow P_0\rightarrow M\rightarrow 0.
\end{align*}
The \textit{$d$th transpose of $M$} is $\Tr_d(M):=\Coker\, (\Hom_\Phi(P_{d-1},\Phi)\rightarrow \Hom_\Phi(P_{d},\Phi))$.
\end{defn}

\begin{proposition}[{\cite[Theorem 3.3.1]{IO}}]\label{thm_dAR_F_endt}
For each non-projective indecomposable object $A^{d+1}$ in $\mathcal{F}$, there exists a $d$-Auslander-Reiten sequence in $\mathcal{F}$ of the form:
\begin{align*}
\xymatrix{\delta: 0\ar[r]&A^0\ar[r]^{\alpha^0}& A^1\ar[r]^{\alpha^1}& A^2\ar[r]&\cdots\ar[r]& A^{d-1}\ar[r]^-{\alpha^{d-1}}& A^d\ar[r]^{\alpha^d}& A^{d+1}\ar[r]&0.
}
\end{align*}
Moreover, if $\delta$ is a $d$-Auslander-Reiten sequence in $\mathcal{F}$, then $A^0=D\Tr_d(A^{d+1})$, where $D(-):=\Hom_k (-,k):\mmod\Phi\rightarrow \mmod\Phi^{op}$ is the standard $k$-duality.
\end{proposition}

\begin{lemma}\label{lemma_nonzero_hdelta}
Let $X$ in $\mathcal{X}$ be an indecomposable such that $\Ext_{\Phi}^d(X,\mathcal{X})$ is non-zero. Suppose $D\Tr_d(X)$ has an $\mathcal{X}$-cover of the form $g:Y\rightarrow D\Tr_d(X)$. Then, for any non-split $d$-exact sequence in $\mathcal{F}$ of the form
 \begin{align*}
    \xymatrix{\delta: 0\ar[r]&X^0\ar[r]^-{\xi^0}& X^1\ar[r]^{\xi^1}&\cdots\ar[r]& X^{d-1}\ar[r]^-{\xi^{d-1}} & X^d\ar[r]^{\xi^d}& X\ar[r]&0,
    }
    \end{align*}
    with all terms in $\mathcal{X}$, there is a morphism $h:X^0\rightarrow Y$ such that $h\cdot \delta$ is a non-split $d$-exact sequence in $\mathcal{F}$. In particular, $\Ext_{\Phi}^d(X,Y)\neq 0$.
\end{lemma}

\begin{proof}
First note that such a $\delta$ exists since $\Ext_{\Phi}^d(X,\mathcal{X})\neq 0$.
Moreover, by Proposition \ref{thm_dAR_F_endt}, there is a $d$-Auslander-Reiten sequence in $\mathcal{F}$ of the form:
    \begin{align*}
    \xymatrix{\epsilon: 0\ar[r]&D \Tr _d (X)\ar[r]^-{\alpha^0}& A^1\ar[r]^{\alpha^1}&\cdots\ar[r]&A^{d-1}\ar[r]^-{\alpha^{d-1}}&  A^d\ar[r]^{\alpha^d}& X\ar[r]&0.
    }
    \end{align*}
    Since $\xi^d$ is not a split epimorphism and $\alpha^d$ is  right almost split in $\mathcal{F}$, there is a morphism $f^d: X^d\rightarrow A^d$ such that $\alpha^d f^d=\xi^d$. Then, by Lemma \ref{lemma_complete_to_morph}, we can construct a morphism of $d$-exact sequences of the form:
    \begin{align*}
    \xymatrix{\delta\ar[d]^f:&0\ar[r] &X^0\ar[r]^{\xi^0}\ar[d]^-{f^0}& X^1\ar[r]^{\xi^1}\ar[d]^-{f^1}& \cdots\ar[r]&X^{d-1}\ar[r]^-{\xi^{d-1}}\ar[d]^-{f^{d-1}}&  X^d\ar[d]^-{f^d}\ar[r]^{\xi^d}& X\ar@{=}[d]\ar[r]&0\\
    \epsilon: & 0\ar[r]&D\Tr_d X\ar[r]_-{\alpha^0}& A^1\ar[r]_{\alpha^1}& \cdots\ar[r]&A^{d-1}\ar[r]_-{\alpha^{d-1}}& A^d\ar[r]_{\alpha^d}& X\ar[r]&0.
    }
    \end{align*}
    Since $g$ is an $\mathcal{X}$-cover, there is a morphism $h:X^0\rightarrow Y$ such that $f^0=g h$. Then, applying $\Ext_{\Phi}^d(X,-)$, we obtain the commutative diagram:
    \begin{align}\label{diagram_ext_comm}
    \begin{gathered}
        \xymatrix{
        \Ext_{\Phi}^d(X,X^0)\ar[rr]^-{\Ext_{\Phi}^d(X,f^0)}\ar[rd]_-{\Ext_{\Phi}^d(X,h)}&&\Ext_{\Phi}^d(X,D\Tr_d(X)).\\
        &\Ext_{\Phi}^d(X,Y)\ar[ru]_-{\Ext_{\Phi}^d(X,g)}&
        }
    \end{gathered}
    \end{align}
    Considering the morphism $\delta\rightarrow f^0\cdot \delta$ obtained as in Remark \ref{remark_Ext_pull_push}(b) and $f:\delta\rightarrow\epsilon$, Lemma \ref{lemma_eq_classes_d-pushout} implies that $0\neq[\epsilon]=[f^0\cdot \delta]$ in $\Ext_{\Phi}^d(X,D\Tr_d(X))$, so that $f^0\cdot \delta$ is non-split by Remark \ref{remark_Ext_pull_push}(a).
    Then, in diagram (\ref{diagram_ext_comm}), we have $\Ext_{\Phi}^{d}(X,gh)(\delta)=gh\cdot \delta=f^0\cdot\delta$ is non-split and so $[h\cdot \delta]\neq 0$, \textbf{i.e.} $h\cdot \delta$ is non-split. In particular $\Ext_{\Phi}^d(X,Y)\neq 0$.
\end{proof}

The argument from \cite[proof of Proposition V.2.1]{ARS} can be easily modified to prove the following higher version. Recall that, following Notation \ref{notation_right_modules}, modules are assumed to be right.

\begin{lemma}\label{lemma_simple_socle}
Let $A$ be an indecomposable non-projective in $\mathcal{F}$. Then we have that $\Ext_{\Phi}^d (A, D \Tr _d (A))$ has a simple socle as an $\End_{\Phi}(A)$-module.
\end{lemma}

\begin{proposition}\label{thm_cover}
\begin{enumerate}[label=(\alph*)]
\item Let $X$ in $\mathcal{X}$ be an indecomposable such that $\Ext_{\Phi}^d(X,\mathcal{X})$ is non-zero. If $D\Tr_d (X)$ has an $\mathcal{X}$-cover of the form $g:Y\rightarrow D\Tr_d(X)$, then $Y=Z\oplus Z'$, where $Z$ is an indecomposable such that $\Ext_{\Phi}^d(X,Z)\neq 0$ and $\Ext_{\Phi}^d(X,Z')=0$. The module $Z$ is unique up to isomorphism.
\item In the setting of (a), a non-split $d$-exact sequence of the form
\begin{align*}
    \xymatrix{
    \epsilon: 0\ar[r]&Y\ar[r]^{\eta^0}& Y^1\ar[r]^{\eta^1}& Y^2\ar[r]^{\eta^2}&\cdots\ar[r]& Y^{d-1}\ar[r]^-{\eta^{d-1}}& Y^d\ar[r]^{\eta^d}& X\ar[r]&0
    }
\end{align*}
is isomorphic to the direct sum of the split $d$-exact sequence:
\begin{align*}
    \xymatrix{
    0\ar[r]&Z'\ar[r]^{1_{Z'}}& Z'\ar[r]& 0\ar[r]&\cdots\ar[r]& 0\ar[r]& 0\ar[r]& 0\ar[r]&0
    }
\end{align*}
and a non-split $d$-exact sequence of the form
\begin{align*}
    \xymatrix{
    0\ar[r]&Z\ar[r]^{\zeta^0}& V\ar[r]^{\zeta^1}& Y^2\ar[r]^{\eta^2}&\cdots\ar[r]& Y^{d-1}\ar[r]^-{\eta^{d-1}}& Y^d\ar[r]^{\eta^d}& X\ar[r]&0.
    }
\end{align*}
\end{enumerate}
\end{proposition}

\begin{proof}
(a) Let $Y=Z_1\oplus\dots\oplus Z_m$ be the indecomposable decomposition of $Y$. By Lemma \ref{lemma_cover_mono}, we have a monomorphism:
    \begin{align*}
    \Ext_{\Phi}^d (X,g): \Ext_{\Phi}^d (X,Y)\,\longrightarrow \Ext_{\Phi}^d (X,D\Tr_d(X)),
    \end{align*}
    which is also a monomorphism of $\End_{\Phi}(X)$-modules. Hence $\im \Ext_{\Phi}^d (X,g)$ is an $\End_{\Phi}(X)$-submodule of $\Ext_{\Phi}^d(X,D\Tr_d(X))$ isomorphic to 
    \begin{align*}
    \Ext_{\Phi}^d(X,Y)\cong\bigoplus_{j=1}^m \Ext_{\Phi}^d(X, Z_j).
    \end{align*}
    Since $\Ext_{\Phi}^d(X,\mathcal{X})\neq 0$, it follows that $X$ is not projective in $\mmod\Phi$. Then, viewed as an $\End_{\Phi}(X)$-module, $\Ext_{\Phi}^d(X, D\Tr_d(X))$ has simple socle by Lemma \ref{lemma_simple_socle}. Hence $\im\Ext_{\Phi}^d(X,g)$ is either zero or an indecomposable $\End_{\Phi}(X)$-module. So there is at most one $j\in\{ 1,\dots,\,m \}$ such that $\Ext_{\Phi}^d(X,Z_j)$ is non-zero.
    Note that $\Ext_{\Phi}^d(X,Y)$ is non-zero by Lemma \ref{lemma_nonzero_hdelta}. Hence there is exactly one $j\in\{ 1,\dots,\,m \}$ such that $\Ext_{\Phi}^d(X,Z_j)$ is non-zero.
   
    (b) By Lemma \ref{lemma_d-pushout_exists}, there is a morphism of $d$-exact sequences of the form:
    \begin{align}\label{diagram_isomorphic_epsilon}
    \begin{gathered}
    \xymatrix{
    \epsilon:\ar[d] 0\ar[r]&Z'\oplus Z\ar[rr]^-{\eta^0=(\xi',\xi)}\ar[d]^{(1,0)}&& Y^1\ar[r]^{\eta^1}\ar[d]& Y^2\ar[r]^{\eta^2}\ar[d]&\cdots\ar[r]^-{\eta^{d-1}}& Y^d\ar[r]^{\eta^d}\ar[d]& X\ar[r]\ar@{=}[d]&0\\
    \overline{\epsilon}:0\ar[r]& Z'\ar[rr]_{\omega^0}&& W^1\ar[r]_{\omega^1}& W^2\ar[r]_{\omega^2}&\cdots\ar[r]_-{\omega^{d-1}}& W^d\ar[r]_{\omega^d}&X\ar[r]&0.
    }
    \end{gathered}
    \end{align}
    Since $\Ext_{\Phi}^d(X,Z')=0$, the bottom row is a split $d$-exact sequence by Remark \ref{remark_Ext_pull_push}(a). Hence, we have that $\overline{\epsilon}$ is isomorphic to:
    \begin{align*}
        \xymatrix{
        0\ar[r]& Z'\ar[rr]^-{\begin{psmallmatrix}1_{Z'}\\0\end{psmallmatrix}}&& Z'\oplus\overline{W^1}\ar[r]^-{(0,\gamma)}& W^2\ar[r]^{\omega^2}&\cdots\ar[r]^-{\omega^{d-1}}& W^d\ar[r]^{\omega^d}&X\ar[r]&0.
        }
    \end{align*}
    Then the morphism (\ref{diagram_isomorphic_epsilon}) is isomorphic to the morphism:
    \begin{align*}
        \xymatrix{
        \epsilon: \ar[d]0\ar[r]&Z'\oplus Z\ar[r]^-{(\xi',\xi)}\ar[d]^{(1,0)}& Y^1\ar[r]^{\eta^1}\ar[d]^-{\begin{psmallmatrix} \mu'\\\mu \end{psmallmatrix}}& Y^2\ar[r]^{\eta^2}\ar[d]&\cdots\ar[r]^-{\eta^{d-1}}& Y^d\ar[r]^{\eta^d}\ar[d]& X\ar[r]\ar@{=}[d]&0\\
        \overline{\epsilon}:0\ar[r]& Z'\ar[r]_-{\begin{psmallmatrix}1_{Z'}\\0\end{psmallmatrix}}& Z'\oplus\overline{W^1}\ar[r]_-{(0,\gamma)}& W^2\ar[r]_{\omega^2}&\cdots\ar[r]_-{\omega^{d-1}}& W^d\ar[r]_{\omega^d}&X\ar[r]&0.
        }
    \end{align*}
    In particular, $\mu'\xi'=1_{Z'}$ and so $Y^1=Z'\oplus V$ and $\epsilon$ isomorphic to a $d$-exact sequence of the form:
    \begin{align*}
        \xymatrix{
        \epsilon: 0\ar[r]&Z'\oplus Z\ar[rr]^-{\begin{psmallmatrix}1_{Z'}& 0\\0&\zeta^0\end{psmallmatrix}}&& Z'\oplus V\ar[r]^-{(0,\zeta^1)}& Y^2\ar[r]^{\eta^2}&\cdots\ar[r]^-{\eta^{d-1}}& Y^d\ar[r]^{\eta^d}& X\ar[r]&0.
        }
    \end{align*}
    Clearly, this is isomorphic to the direct sum of the two $d$-exact sequences we wanted, where the one starting at $Z$ does not split since $\epsilon$ does not split.
\end{proof}

\begin{defn}
Suppose that $\mathcal{X}$ is precovering in $\mathcal{F}$ and let $X$ be an indecomposable in $\mathcal{X}$. If $\Ext_{\Phi}^d(X,\mathcal{X})=0$ we put $\sigma X=0$. Otherwise, letting $g:Y\rightarrow D\Tr_d(X)$ be an $\mathcal{X}$-cover,  we denote by $\sigma X$ the unique indecomposable direct summand $Z$ of $Y$ such that $\Ext_{\Phi}^d(X,Z)\neq 0$.
\end{defn}

\begin{corollary}
Let $\mathcal{X}$ be precovering in $\mathcal{F}$ and let
\begin{align*}
    \xymatrix{\delta: 0\ar[r]&X^0\ar[r]^-{\xi^0}& X^1\ar[r]^{\xi^1}&\cdots\ar[r]& X^{d-1}\ar[r]^-{\xi^{d-1}} & X^d\ar[r]^{\xi^d}& X\ar[r]&0
    }
\end{align*}
be a $d$-Auslander-Reiten sequence in $\mathcal{X}$. Then $X^0\cong \sigma X$.
\end{corollary}

\begin{proof}
Note that the existence of $\delta$ implies that $\Ext_{\Phi}^d(X,\mathcal{X})$ is non-zero.
As $\mathcal{X}$ is precovering in $\mathcal{F}$, there is an $\mathcal{X}$-cover $g:Y\rightarrow D\Tr_d(X)$. Then, by Lemma \ref{lemma_nonzero_hdelta}, there is a morphism of non-split $d$-exact sequences in $\mathcal{F}$ of the form:
 \begin{align*}
    \xymatrix{\delta:\ar[d]&0\ar[r] &X^0\ar[r]^{\xi^0}\ar[d]^-{h}& X^1\ar[r]^{\xi^1}\ar[d]^-{h^1}& \cdots\ar[r]&X^{d-1}\ar[r]^-{\xi^{d-1}}\ar[d]^-{h^{d-1}}&  X^d\ar[d]^-{h^d}\ar[r]^{\xi^d}& X\ar@{=}[d]\ar[r]&0\\
    h\cdot \delta: & 0\ar[r]&Y\ar[r]_{\eta^0}& Y^1\ar[r]_{\eta^1}& \cdots\ar[r]&Y^{d-1}\ar[r]_-{\eta^{d-1}}& Y^d\ar[r]_{\eta^d}& X\ar[r]&0.
    }
\end{align*}
Since $\eta^d$ is not a split epimorphism, Lemma \ref{lemma_null-homotopic} implies that $h$ does not factor through $\xi^0$. As $\xi^0$ is a left almost split morphism in $\mathcal{X}$, it follows that $h$ is a split monomorphism. Hence  $X^0$ is an indecomposable direct summand of $Y$ such that $\Ext_{\Phi}^d (X,X^0)\neq 0$ and Proposition \ref{thm_cover}(a) implies that $X^0\cong \sigma X$.
\end{proof}

\begin{lemma}\label{lemma_defect_subf}
Any $d$-exact sequence in $\mathcal{F}$ of the form:
\begin{align*}
\xymatrix{\delta: 0\ar[r]&A^0\ar[r]^{\alpha^0}& A^1\ar[r]^{\alpha^1}& A^2\ar[r]&\cdots\ar[r]& A^{d-1}\ar[r]^-{\alpha^{d-1}}& A^d\ar[r]^{\alpha^d}& A^{d+1}\ar[r]&0,
}
\end{align*}
induces the exact sequences
\begin{align*}
\xymatrix@C=1em{
0\ar[r]& (B,A^0)\ar[r]&\cdots\ar[r]&(B, A^d)\ar[r] &(B,A^{d+1})\ar[r]&\Ext_{\Phi}^d(B,A^0)\ar[r]&\Ext_{\Phi}^d(B,A^1),
\\
0\ar[r]& (A^{d+1},B)\ar[r]&\cdots\ar[r]& (A^{1},B)\ar[r]&(A^0,B)\ar[r]&\Ext_{\Phi}^d(A^{d+1},B)\ar[r]&\Ext_{\Phi}^d(A^d,B),
}
\end{align*}
for any $B$ in $\mathcal{F}$.
\end{lemma}

\begin{proof}
See \cite[Proposition 2.2]{JK}.
\end{proof}

\begin{defn}{\cite[Definition 3.1]{JK}}\label{defn_defect}
Consider a $d$-exact sequence in $\mathcal{F}$ of the form:
\begin{align*}
\xymatrix{\delta: 0\ar[r]&A^0\ar[r]^{\alpha^0}& A^1\ar[r]^{\alpha^1}& A^2\ar[r]&\cdots\ar[r]& A^{d-1}\ar[r]^-{\alpha^{d-1}}& A^d\ar[r]^{\alpha^d}& A^{d+1}\ar[r]&0.
}
\end{align*}
We define $\delta^*$, the \textit{contravariant defect of $\delta$ on $\mathcal{F}$}, by the exact sequence of functors
\begin{align*}
    (-,A^d)\rightarrow (-,A^{d+1})\rightarrow \delta^*(-)\rightarrow 0.
\end{align*}
Dually, we define $\delta_*$, the \textit{covariant defect of $\delta$ on $\mathcal{F}$}, by the exact sequence of functors
\begin{align*}
    (A^1,- )\rightarrow (A^0,-)\rightarrow \delta_*(-)\rightarrow 0.
\end{align*}
\end{defn}

\begin{remark}\label{remark_subfunctor_defect}
Note that, by Lemma \ref{lemma_defect_subf}, we have that $\delta^{*}(-)$ is a subfunctor of $\Ext_{\Phi}^d(-,A^0)\mid_{\mathcal{F}}$ and $\delta_*(-)$ is a subfunctor of $\Ext_{\Phi}^d(A^{d+1},-)\mid_{\mathcal{F}}$.
\end{remark}

\begin{lemma}\label{lemma_dimension_defect}
Consider a $d$-exact sequence in $\mathcal{F}$ with all terms in $\mathcal{X}$ of the form:
\begin{align*}
    \xymatrix{\delta: 0\ar[r]&X^0\ar[r]^-{\xi^0}& X^1\ar[r]^{\xi^1}&\cdots\ar[r]& X^{d-1}\ar[r]^-{\xi^{d-1}} & X^d\ar[r]^{\xi^d}& X^{d+1}\ar[r]&0
    }
\end{align*}
and an $\mathcal{X}$-cover $g:X\rightarrow A$ for some $A\in\mathcal{F}$. The $k$-linear map $(X^0,g):(X^0,X)\rightarrow (X^0,A)$ induces an isomorphism of $k$-vector spaces:
\begin{align*}
    \delta_*(g):\delta_*(X)\xrightarrow{\sim}\delta_*(A).
\end{align*}
In particular, $\dim_k (\delta_*(X))=\dim_k (\delta_*(A))$.
\end{lemma}

\begin{proof}
By Definition \ref{defn_defect}, we have
\begin{align*}
    \delta_*(g):(X^0,X)/\im(\xi^0,X)\rightarrow (X^0,A)/\im(\xi^0,A).
\end{align*}
Note that since $g:X\rightarrow A$ is an $\mathcal{X}$-cover, the map $(X^0,g)$ is surjective. Hence it is enough to show that $\im(\xi^0,X)$ is the full preimage of $\im(\xi^0,A)$ under $(X^0,g)$. It is clear that
\begin{align*}
    (X^0,g)(\im(\xi^0, X))\subseteq\im(\xi^0,A).
\end{align*}
It remains to show that if $h:X^0\rightarrow X$ is such that $gh:X^0\rightarrow A$ factors through $\xi^0$, then $h$ factors through $\xi^0$. Consider the following morphisms of $d$-exact sequences:
 \begin{align*}
    \xymatrix@C=2em{\delta\ar[d]:&0\ar[r] &X^0\ar[r]^{\xi^0}\ar[d]_-{h}& X^1\ar[r]^{\xi^1}\ar[ddl]\ar[d]& \cdots\ar[r]&X^{d-1}\ar[r]^-{\xi^{d-1}}\ar[d]&  X^d\ar[d]\ar[r]^{\xi^d}& X^{d+1}\ar@{=}[d]\ar[r]\ar@{-->}[ddl]&0\\
    h\cdot \delta:\ar[d] & 0\ar[r]&X\ar[r]\ar[d]_g& Y^1\ar[r]\ar[d]& \cdots\ar[r]&Y^{d-1}\ar[r]\ar[d]& Y^d\ar[r]\ar[d]& X^{d+1}\ar[r]\ar@{=}[d]&0\\
    gh\cdot \delta: & 0\ar[r]&A\ar[r]& A^1\ar[r]& \cdots\ar[r]&A^{d-1}\ar[r]& A^d\ar[r]& X^{d+1}\ar[r]&0.
    }
\end{align*}
Since $gh$ factors through $\xi^0$, Lemma \ref{lemma_null-homotopic} implies that the bottom row splits. Hence, we have that $[\Ext_{\Phi}^d(X^{d+1},g)(h\cdot\delta)]=0$. Since $\Ext_{\Phi}^d(X^{d+1},g)$ is a monomorphism by Lemma \ref{lemma_cover_mono}, it follows that the middle row splits. Hence $h$ factors through $\xi^0$ by Lemma \ref{lemma_null-homotopic}.
\end{proof}

\begin{remark}\label{remark_defect_dim}
Let $X\in\mathcal{X}$ be indecomposable and assume that $D\Tr_d(X)$ has an $\mathcal{X}$-cover, say $g:Y\rightarrow D\Tr_d(X)$. Given any $d$-exact sequence with terms in $\mathcal{X}$ of the form
\begin{align*}
     \xymatrix{\delta: 0\ar[r]&X^0\ar[r]^-{\xi^0}& X^1\ar[r]^{\xi^1}&\cdots\ar[r]& X^{d-1}\ar[r]^-{\xi^{d-1}} & X^d\ar[r]^{\xi^d}& X\ar[r]&0,
    }
\end{align*}
we have that
\begin{align*}
     \dim_k(\delta_*(Y))=\dim_k(\delta_*(D\Tr_d(X)))=\dim_k(\delta^*(X)),
\end{align*}
where the first equality holds by Lemma \ref{lemma_dimension_defect} and the second by \cite[Theorem 3.8]{JK}.
\end{remark}

\begin{proposition}\label{thm_ras_X}
Assume $\mathcal{X}$ is precovering in $\mathcal{F}$. Let $X\in\mathcal{X}$ be an indecomposable such that $\Ext_{\Phi}^d(X,\mathcal{X})\neq 0$ and $g:Y\rightarrow D\Tr_d(X)$ be an $\mathcal{X}$-cover. Then there is a $d$-exact sequence with terms in $\mathcal{X}$ of the form:
\begin{align*}
    \xymatrix{
    \epsilon: 0\ar[r]&Y\ar[r]^{\eta^0}& Y^1\ar[r]^{\eta^1}& Y^2\ar[r]^{\eta^2}&\cdots\ar[r]& Y^{d-1}\ar[r]^-{\eta^{d-1}}& Y^d\ar[r]^{\eta^d}& X\ar[r]&0,
    }
\end{align*}
with $\eta^d$ right almost split in $\mathcal{X}$.
\end{proposition}

\begin{proof}
Since $\Ext_{\Phi}^d(X,\mathcal{X})\neq 0$, there exists a non-split $d$-exact sequence with terms from $\mathcal{X}$ of the form:
\begin{align*}
    \xymatrix{\delta: 0\ar[r]&X^0\ar[r]^-{\xi^0}& X^1\ar[r]^{\xi^1}&\cdots\ar[r]& X^{d-1}\ar[r]^-{\xi^{d-1}} & X^d\ar[r]^{\xi^d}& X\ar[r]&0.
    }
\end{align*}
As not every endomorphism of $X$ factors through $\xi^d$, we have that $\dim_k(\delta^*(X))\neq 0$. By Remark \ref{remark_defect_dim}, we have that
\begin{align*}
    \dim_k(\delta_*(Y))=\dim_k(\delta_*(D\Tr_d(X)))=\dim_k(\delta^*(X))\neq 0.
\end{align*}
So $\Ext_{\Phi}^d(\mathcal{X},Y)$ is non-zero by Remark \ref{remark_subfunctor_defect} and there is a non-split $d$-exact sequence with terms in $\mathcal{X}$ of the form:

\begin{align*}
    \xymatrix{\zeta: 0\ar[r]&Y\ar[r]^-{\zeta^0}& Z^1\ar[r]^{\zeta^1}&\cdots\ar[r]& Z^{d-1}\ar[r]^-{\zeta^{d-1}} & Z^d\ar[r]^{\zeta^d}& Z^{d+1}\ar[r]&0.
    }
\end{align*}
Since not every endomorphism of $Y$ factors through $\zeta^0$, then $\dim_k(\zeta_*(Y))$ is non-zero and so, by Remark \ref{remark_defect_dim}, we have 
\begin{align*}
   0\neq \dim_k(\zeta_*(Y))=\dim_k(\zeta_*(D\Tr_d(X)))=\dim_k(\zeta^*(X)).
\end{align*}
Hence not every morphism of the form $X\rightarrow Z^{d+1}$ factors through $\zeta^d$. So there is a morphism $t:X\rightarrow Z^{d+1}$ such that its image in $\zeta^*(X)=(X, Z^{d+1})/\im(X,\zeta^d)$ generates a simple $\End_{\Phi}(X)$-module. Thus, by the dual of Remark \ref{remark_Ext_pull_push}(b), we have a morphism of $d$-exact sequences in $\mathcal{F}$ of the form:
\begin{align*}
    \xymatrix{\zeta\cdot t:\ar[d]&0\ar[r] &Y\ar[r]^{\eta^0}\ar@{=}[d]& Y^1\ar[r]^{\eta^1}\ar[d]^-{t^1}& \cdots\ar[r]&Y^{d-1}\ar[r]^-{\eta^{d-1}}\ar[d]^-{t^{d-1}}&  Y^d\ar[d]^-{t^d}\ar[r]^{\eta^d}& X\ar[d]^{t}\ar[r]&0\\
    \zeta: & 0\ar[r]&Y\ar[r]_-{\zeta^0}& Z^1\ar[r]_{\zeta^1}&\cdots\ar[r]& Z^{d-1}\ar[r]_-{\zeta^{d-1}} & Z^d\ar[r]_-{\zeta^d}& Z^{d+1}\ar[r]&0,
    }
    \end{align*}
where we can assume $Y^1,\dots,\, Y^d$ are in $\mathcal{X}$ by Remark \ref{remark_almost_minimal}.
We claim that $\epsilon:=\zeta\cdot t$ is such that $\eta^d$ is right almost split in $\mathcal{X}$. First note that since $t$ does not factor through $\zeta^d$, then $\epsilon$ is not a split $d$-exact sequence by Lemma \ref{lemma_null-homotopic}. In particular, $\eta^d$ is not a split epimorphism. Suppose that $s:W\rightarrow X$ in $\mathcal{X}$ is not a split epimorphism. We need to show that $s$ factors through $\eta^d$. Consider the morphism obtained by the dual of Remark \ref{remark_Ext_pull_push}(b):
\begin{align*}
    \xymatrix{\epsilon\cdot s:\ar[d]&0\ar[r] &Y\ar[r]^{\omega^0}\ar@{=}[d]& W^1\ar[r]^{\omega^1}\ar[d]^-{s^1}& \cdots\ar[r]&W^{d-1}\ar[r]^-{\omega^{d-1}}\ar[d]^-{s^{d-1}}&  W^d\ar[d]^-{s^d}\ar[r]^-{\omega^d}& W\ar[d]^{s}\ar[r]&0\\
    \epsilon: & 0\ar[r]&Y\ar[r]_-{\eta^0}& Y^1\ar[r]_{\eta^1}&\cdots\ar[r]& Y^{d-1}\ar[r]_-{\eta^{d-1}} & Y^d\ar[r]_-{\eta^d}& X\ar[r]&0.
    }
    \end{align*}
By Lemma \ref{lemma_null-homotopic}, we have that $s$ factoring through $\eta^d$ is equivalent to $\epsilon\cdot s$ splitting. By Remark \ref{remark_defect_dim}, it is enough to show that every morphism $r:X\rightarrow W$  factors through $\omega^d$. Note that since $s$ is not a split epimorphism, $sr:X\rightarrow X$ is not an isomorphism. Hence, $tsr:X\rightarrow Z^{d+1}$ is in $t\End_{\Phi}(X)\rad_{\End_{\Phi}(X)}$. Since the image of $t\End_{\Phi}(X)$ in $\zeta^*(X)$ is a simple module, it follows that $tsr$ projects to zero in $\zeta^*(X)$. In other words, $tsr$ factors through $\zeta^d$, so there is a morphism $\alpha: X\rightarrow Z^d$ such that $\zeta^d\alpha=tsr$. Consider:
\begin{align*}
    \xymatrix{\epsilon\cdot sr:\ar[d]&0\ar[r] &Y\ar[r]^{\mu^0}\ar@{=}[d]& U^1\ar[r]^{\mu^1}\ar[d]^-{r^1}\ar@{-->}[ldd]& \cdots\ar[r]&U^{d-1}\ar[r]^-{\mu^{d-1}}\ar[d]^-{r^{d-1}}&  U^d\ar[d]^-{r^d}\ar[r]^{\mu^d}& X\ar[d]^{r}\ar[r]\ar[ldd]&0\\
    \epsilon\cdot s:\ar[d]&0\ar[r] &Y\ar[r]^{\omega^0}\ar@{=}[d]& W^1\ar[r]^{\omega^1}\ar[d]^-{t^1s^1}& \cdots\ar[r]&W^{d-1}\ar[r]^-{\omega^{d-1}}\ar[d]_-{t^{d-1}s^{d-1}}&  W^d\ar[d]_-{t^d s^d}\ar[r]^-{\omega^d}& W\ar[d]^{ts}\ar[r]&0\\
    \zeta: & 0\ar[r]&Y\ar[r]_-{\zeta^0}& Z^1\ar[r]_{\zeta^1}&\cdots\ar[r]& Z^{d-1}\ar[r]_-{\zeta^{d-1}} & Z^d\ar[r]_-{\zeta^d}& Z^{d+1}\ar[r]&0.
    }
    \end{align*}
Then, by Lemma \ref{lemma_null-homotopic}, there is a morphism $\alpha^1:U^1\rightarrow Y$ such that $\alpha^1\mu^0=1_Y$. Hence the top row of the above diagram splits. So there is a morphism $\phi:X\rightarrow U^d$ such that $\mu^d\phi=1_X$. Note that
\begin{align*}
    \omega^d r^d\phi=r\mu^d\phi=r.
\end{align*}
Hence $r$ factors through $\omega^d$ as we wished to prove.
\end{proof}

\begin{theorem}\label{coro_d-ARseq}
Assume $\mathcal{X}$ is precovering in $\mathcal{F}$ and let $X$ be an indecomposable in $\mathcal{X}$.
\begin{enumerate}[label=(\alph*)]
    \item There exists a right almost split morphism $W\rightarrow X$ in $\mathcal{X}$.
    \item If $\Ext_{\Phi}^d(X,\mathcal{X})$ is non-zero, there is a $d$-Auslander-Reiten sequence in $\mathcal{X}$ of the form:
    \begin{align}\label{diagram_dAR_X}
    \xymatrix{0\ar[r]&\sigma X\ar[r]^-{\xi^0}& X^1\ar[r]^{\xi^1}&\cdots\ar[r]& X^{d-1}\ar[r]^-{\xi^{d-1}} & X^d\ar[r]^-{\xi^d}& X\ar[r]&0.
    }
\end{align}
\end{enumerate}
\end{theorem}
\begin{proof}
(a) This follows from \cite[Proposition\ 3.10]{AMSS}.
    
(b) Let $g:Y\rightarrow D\Tr_d (X)$ be an $\mathcal{X}$-cover. Then, by Proposition \ref{thm_ras_X}, there exists a $d$-exact sequence with terms in $\mathcal{X}$ of the form
    \begin{align*}
    \xymatrix{
    \epsilon: 0\ar[r]&Y\ar[r]^{\eta^0}& Y^1\ar[r]^{\eta^1}& Y^2\ar[r]^{\eta^2}&\cdots\ar[r]& Y^{d-1}\ar[r]^-{\eta^{d-1}}& Y^d\ar[r]^{\eta^d}& X\ar[r]&0,
    }
\end{align*}
with $\eta^d$ right almost split in $\mathcal{X}$. By Proposition \ref{thm_cover}, $\epsilon$ has a non-split $d$-exact sequence with terms in $\mathcal{X}$ of the form
\begin{align*}
    \xymatrix{
    \delta:0\ar[r]&\sigma X \ar[r]^{\zeta^0}& V\ar[r]^{\zeta^1}& Y^2\ar[r]^{\eta^2}&\cdots\ar[r]& Y^{d-1}\ar[r]^-{\eta^{d-1}}& Y^d\ar[r]^{\eta^d}& X\ar[r]&0
    }
\end{align*}
as a direct summand. If $d\geq 2$, we may also assume that $\zeta^1,\,\eta^2,\dots,\,\eta^{d-1}$ are in $\rad_\mathcal{X}$. Moreover, since $\sigma X$ is indecomposable and $\zeta^0$ is not a split monomorphism, it follows that $\zeta^0$ is in $\rad_\mathcal{X}$. Hence, by Lemma \ref{lemma_dAR_equiv}, we conclude that $\delta$ is a $d$-Auslander-Reiten sequence in $\mathcal{X}$.
\end{proof}

\section{More on $d$-Auslander-Reiten sequences in $\mathcal{X}$ and the case when $\underline{\End}_\Phi(X)$ is a division ring}\label{section_6}

In this section, we study the case when, for an indecomposable $X\in\mathcal{X}$, the factor ring of $\End_\Phi (X)$ modulo the morphisms factoring through a projective is a division ring. Generalising \cite[Corollary V.2.4]{ARS}, we prove that an almost minimal $d$-exact sequence with terms in $\mathcal{X}$ ending at $X$ is a $d$-Auslander-Reiten sequence if and only if it does not split. As a consequence of this result, we prove a higher version of \cite[Proposition 2.10]{MK}.

The argument from \cite[proof of Proposition V.2.1]{ARS} can be easily modified to prove the following result. Note that this differs from the original result in two ways: it is a higher version and we work in the subcategory $\mathcal{X}$. The condition on an indecomposable $C$ in $\mmod\Phi$ to be non-projective is hence substituted with the condition on an indecomposable $X\in\mathcal{X}$ to be such that $\Ext_\Phi^d(X,\mathcal{X})\neq 0$ and $D\Tr C$ with $\sigma X$.

\begin{lemma}\label{lemma_higher_X_ARS}
Let $X$ be an indecomposable in $\mathcal{X}$ such that $\Ext_\Phi^d(X,\mathcal{X})\neq 0$. Then $\Ext^d(X, \sigma X)$, as an $\End_{\Phi}(X)$-module, has a simple socle generated by a $d$-Auslander-Reiten sequence in $\mathcal{X}$ of the form (\ref{diagram_dAR_X}).
\end{lemma}

\begin{lemma}\label{lemma_ARS_gen}
Assume $\mathcal{X}$ is precovering in $\mathcal{F}$. Let $X$ be an indecomposable in $\mathcal{X}$ such that $\Ext_{\Phi}^d(X,\mathcal{X})\neq 0$. Consider a  non-split $d$-exact sequence of the form:
\begin{align*}
    \xymatrix{\delta: 0\ar[r]&\sigma X\ar[r]^-{\xi^0}& X^1\ar[r]^{\xi^1}&\cdots\ar[r]& X^{d-1}\ar[r]^-{\xi^{d-1}} & X^d\ar[r]^{\xi^d}& X\ar[r]&0,
    }
\end{align*}
with $X^1,\dots,\, X^d$ in $\mathcal{X}$ and, when $d\geq 2$, also $\xi^1,\dots,\, \xi^{d-1}$ in $\rad_{\mathcal{X}}$. Then, the following are equivalent:
 \begin{enumerate}[label=(\alph*)]
     \item $\delta$ is a $d$-Auslander-Reiten sequence in $\mathcal{X}$,
     \item $\xi^d$ is right almost split in $\mathcal{X}$,
     \item $\im (X,\xi^d)=\rad_{\End_{\Phi}(X)}$,
     \item $\delta^*(X)$ is a simple $\End_\Phi(X)$-module.
 \end{enumerate}
\end{lemma}

\begin{proof}
By Definition \ref{defn_dAR_seq}, we have that (a) implies (b).
Assume now that (b) holds and note that since $X$ is indecomposable, then $\End_\Phi(X)$ is local.  Consider
\begin{align*}
    (X,\xi^d): (X, X^d)\rightarrow (X,X): \alpha\mapsto \xi^d\alpha.
\end{align*}
Assume $\beta:X\rightarrow X$ is in $\rad_{\End_\Phi(X)}$.Then, since $X$ is indecomposable, it follows that $\beta$ is not an isomorphism and so $\beta$ is not a split epimorphism. As $\xi^d$ is right almost split in $\mathcal{X}$, there exists a morphism $\alpha:X\rightarrow X^d$ such that
\begin{align*}
    \beta=\xi^d\alpha=(X,\xi^d)(\alpha),
\end{align*}
and so $\beta\in\im(X,\xi^d)$. Assume now that $\beta:X\rightarrow X$ is in $\im(X,\xi^d)$, \textbf{i.e.} $\beta=\xi^d\alpha$ for some $\alpha\in(X,X)$. Then, since $\xi^d$ is not a split epimorphism, it follows that $\beta$ is not an isomorphism and so $\beta$ is in $\rad_{\End_\Phi(X)}$. Hence (b) implies (c).

Recall that $\delta^*(X)=(X,X)/\im (X,\xi^d)$. Assume (c) holds. Then we have that $\delta^*(X)=\End_\Phi(X)/\rad_{\End_\Phi(X)}$ and this is simple as $\rad_{\End_\Phi(X)}$ is maximal. So (c) implies (d).

Assume now that (d) holds. Then, by Lemma \ref{lemma_higher_X_ARS}, we have that $\delta^*(X)$ is the socle of $\Ext^d_\Phi(X,\sigma X)$ as an $\End_\Phi(X)$-module and the non-split $d$-exact sequence $\delta$ is a $d$-Auslander-Reiten sequence in $\mathcal{X}$. So (d) implies (a).
\end{proof}

\begin{notation}
For a module $A$ in $\mathcal{F}$, we denote by $\mathcal{P}(A)$ the ideal of all morphisms of the form $A\rightarrow A$ that factor through a projective module. The factor ring of $\End_\Phi (A)$ modulo $\mathcal{P}(A)$ is then denoted by $\underline{\End}_\Phi(A)$.
\end{notation}

\begin{theorem}\label{coro_ARS_gen}
Assume $\mathcal{X}$ is precovering in $\mathcal{F}$. Let $X$ be an indecomposable in $\mathcal{X}$ such that $\underline{\End}_\Phi(X)$ is a division ring. For a $d$-exact sequence of the form:
\begin{align*}
    \xymatrix{\delta: 0\ar[r]&\sigma X\ar[r]^-{\xi^0}& X^1\ar[r]^{\xi^1}&\cdots\ar[r]& X^{d-1}\ar[r]^-{\xi^{d-1}} & X^d\ar[r]^{\xi^d}& X\ar[r]&0,
    }
\end{align*}
with terms in $\mathcal{X}$ and, when $d\geq 2$, also $\xi^1,\dots,\,\xi^{d-1}$ in $\rad_{\mathcal{X}}$, the following are equivalent:
\begin{enumerate}[label=(\alph*)]
    \item $\delta$ is a $d$-Auslander-Reiten sequence in $\mathcal{X}$,
    \item $\delta$ does not split.
\end{enumerate}
\end{theorem}

\begin{proof}
Note that as $\xi^d$ is an epimorphism, $\im (X,\xi^d)$ contains $\mathcal{P}(X)$. Since $\underline{\End}_\Phi(X) =\End_\Phi(X) /\mathcal{P}(X)$ is a division ring, it is simple as an $\End_\Phi(X) $-module. Then $\mathcal{P}(X)$ is a maximal submodule of $\End_\Phi(X) $ and, as $\End_\Phi(X) $ is local, we have that $\mathcal{P}(X)=\rad_{\End_\Phi(X) }$.  Hence, maximality and
\begin{align*}
    \rad_{\End_\Phi(X) }=\mathcal{P}(X)\subseteq \im (X,\xi^d)\subseteq \End_\Phi(X) ,
\end{align*}
imply that we have the following two cases:
\begin{enumerate}
    \item $\im (X,\xi^d)=\rad_{\End_\Phi(X)}$, \textbf{i.e.} $\delta^*(X)$ is a simple $\End_\Phi(X)$-module, in which case $\delta$ is non-split as $1_X\not\in\im (X,\xi^d)$;
    \item $\im(X,\xi^d)=\End_\Phi(X)$, \textbf{i.e.} $\delta^*(X)=0$ is not a simple $\End_\Phi(X)$-module, in which case $\delta$ splits as $1_X\in\im (X,\xi^d)$.
\end{enumerate}
Hence $\delta^*(X)$ is a simple $\End_\Phi(X)$-module if and only if $\delta$ does not split. Then, by Lemma \ref{lemma_ARS_gen}, we conclude that $\delta$ does not split if and only if $\delta$ is a $d$-Auslander-Reiten sequence in $\mathcal{X}$.
\end{proof}

\begin{corollary}\label{coro_dAR_division}
Assume $\mathcal{X}$ is precovering in $\mathcal{F}$. Let $g:Y\rightarrow D\Tr_d (X)$ be an $\mathcal{X}$-cover, where $X$ is an indecomposable in $\mathcal{X}$ with $\underline{\End}_\Phi(X)$ a division ring. Consider a non-split $d$-exact sequence with terms in $\mathcal{X}$ of the form:
\begin{align*}
    \xymatrix{
\epsilon:&0\ar[r] &Y\ar[r]^{\eta^0}& Y^1\ar[r]^{\eta^1}& \cdots\ar[r]&  Y^d\ar[r]^{\eta^d}& X\ar[r]&0,}
\end{align*}
where, if $d\geq 2$, we also have $\eta^1,\dots,\,\eta^{d-1}\in\rad_{\mathcal{X}}$. Consider a morphism induced by a $d$-pushout diagram:
\begin{align*}
\xymatrix{
\epsilon:\ar[d]&0\ar[r] &Y\ar[r]^{\eta^0}\ar[d]^-{g}& Y^1\ar[r]^{\eta^1}\ar[d]^-{g^1}& \cdots\ar[r]&  Y^d\ar[d]^-{g^d}\ar[r]^{\eta^d}& X\ar@{=}[d]\ar[r]&0\\
 \delta: & 0\ar[r]&D\Tr_d (X)\ar[r]_-{\alpha^0}& A^1\ar[r]_{\alpha^1}& \cdots\ar[r]& A^d\ar[r]_{\alpha^d}& X\ar[r]&0,
}
\end{align*}
where, if $d\geq 2$, we have that $\alpha^1,\dots,\,\alpha^{d-1}\in\rad_{\mathcal{F}}$. Then $\delta$ is a $d$-Auslander-Reiten sequence in $\mathcal{F}$ and $\eta^d$ is right almost split in $\mathcal{X}$.
\end{corollary}

\begin{proof}
First note that in a $d$-pushout diagram of $\epsilon$ along $g$,  the middle morphisms $\alpha^1,\cdots,\, \alpha^{d-1}$ are not necessarily in $\rad_{\mathcal{F}}$. However, dropping extra direct summands of the form $A\xrightarrow{\cong} A$, we obtain a $d$-pushout diagram with middle morphisms in $\rad_{\mathcal{F}}$.

Considering Theorem \ref{coro_ARS_gen} in the case when $\mathcal{X}=\mathcal{F}$, so that $\sigma X=D\Tr_d(X)$, we have that  $\delta$ is a $d$-Auslander-Reiten sequence in $\mathcal{F}$ if it does not split.
Suppose for a contradiction that $\delta$ is a split $d$-exact sequence. Then Lemma \ref{lemma_null-homotopic} implies that there is a morphism $h:Y^1\rightarrow D\Tr_d(X)$ such that $h\eta^0=g$. Moreover, since $Y^1\in\mathcal{X}$ and $g$ is an $\mathcal{X}$-cover, there is a morphism $\phi:Y^1\rightarrow Y$ such that $h=g\phi$. Hence
\begin{align*}
    g=h\eta^0=g\phi\eta^0,
\end{align*}
and $\phi\eta^0$ is an isomorphism as $g$ is right minimal. But this implies that $\eta^0$ is a split monomorphism, contradicting our initial assumption. So $\delta$ does not split.

By Proposition \ref{thm_cover}(b), we have that $\epsilon$ is isomorphic to the direct sum of a split $d$-exact sequence:
\begin{align*}
    \xymatrix{
    0\ar[r]& Y'\ar[r]^{1_{Y'}}& Y'\ar[r]& 0\ar[r]&\cdots\ar[r]&0\ar[r]&0\ar[r]&0
    }
\end{align*}
and a non-split $d$-exact sequence:
\begin{align*}
    \xymatrix{
\zeta:&0\ar[r] &\sigma X\ar[r]^-{\zeta^0}& W\ar[r]^-{\zeta^1}& Y^2\ar[r]^{\eta^2}& \cdots\ar[r]^{\eta^{d-1}}&  Y^d\ar[r]^{\eta^d}& X\ar[r]&0,}
\end{align*}
where, for $d\geq 2$, we have that $\zeta^1,\, \eta^2,\dots,\eta^{d-1}$ are in $\rad_{\mathcal{X}}$.
Note that, by Theorem \ref{coro_ARS_gen}, we have that $\zeta$ is a $d$-Auslander-Reiten sequence in $\mathcal{X}$ and in particular $\eta^d$ is a right almost split morphism in $\mathcal{X}$.
\end{proof}

\section{Example}\label{section_7}
In this section, we illustrate the results from Section \ref{section_Kleiner2} to a $2$-representation finite algebra $\Phi$. Here we assume that $\Phi$ is an algebra over an algebraically closed field $k$ in order to be able to apply \cite[Theorem B]{HJV}.

\begin{defn}[{\cite[Definition 2.2]{IOO}}]
The algebra $\Phi$ is called \textit{$d$-representation finite} if gldim$\Phi\leq d$ and $\Phi$ has a $d$-cluster tilting object.
\end{defn}

Let $\Phi$ be the algebra defined by the following quiver with relations:

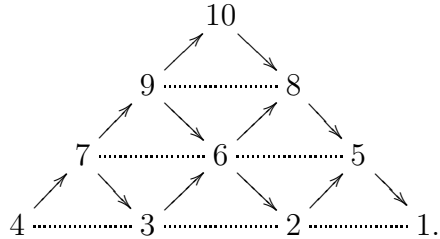
\begin{align*}
\xymatrix @C=1em@R=1em{
&&& 10\ar[rd] &&&\\
&& 9\ar[ru]\ar[rd]\ar@{..}[rr] && 8\ar[rd] &&\\
& 7 \ar[ru]\ar[rd]\ar@{..}[rr] && 6\ar[ru]\ar[rd]\ar@{..}[rr] && 5\ar[rd]&\\
4\ar[ru]\ar@{..}[rr] && 3\ar[ru]\ar@{..}[rr] && 2\ar[ru]\ar@{..}[rr] && 1.
}
\end{align*}

\begin{remark}
Note that the algebra $\Phi$ is $2$-representation finite by \cite[Theorem\ 1.18]{I}. Moreover, by \cite[Theorem 1.6]{I}, we have that $\mmod\Phi$ has the unique $2$-cluster tilting subcategory
\begin{align*}
    \mathcal{F}=\add \{ (D\Tr_2)^{j} (i)\mid i \text{ injective in } \mmod\Phi \text{ and } j\geq 0 \} .
\end{align*}
\end{remark}
Denoting the indecomposable modules in $\mmod\Phi$ by their radical series, we find the Auslander-Reiten quiver of $\mathcal{F}$ is as illustrated in Figure \ref{fig:AR_F}, see \cite[Theorems 3.3 and 3.4]{OT},
where the dashed arrows show the action of $D\Tr_2$.
\begin{figure}
\begin{align*}
\xymatrix @!0 @R=3.5em @C=3.5em{
&&& \color{red}{\begin{smallmatrix}10\\8\\5\\1\end{smallmatrix}}\ar[rd] &&&\\
&& \color{red}{\begin{smallmatrix}8\\5\\1\end{smallmatrix}}\ar[ru]\ar[rd] && \color{red}{\begin{smallmatrix}&9&\\10&&6\\&8&&2\\&&5\end{smallmatrix}}\ar[rd]\ar[rddd] &&\\
& {\begin{smallmatrix}5\\1\end{smallmatrix}}\ar[ru]\ar[rd] && \color{red}{\begin{smallmatrix}&6&\\2&&8\\&5&\end{smallmatrix}}\ar[ru]\ar[rd]\ar[rddd] && {\begin{smallmatrix}&7&\\3&&9\\&6&&10\\&&8\end{smallmatrix}}\ar[rd]\ar[rddd]&\\
\color{red}{\begin{smallmatrix}1\end{smallmatrix}}\ar[ru] && {\begin{smallmatrix}2\\5\end{smallmatrix}} \ar[ru]\ar[rddd] && {\begin{smallmatrix}3\\6\\8\end{smallmatrix}}\ar[ru]\ar[rddd] && \color{red}{\begin{smallmatrix}4\\7\\9\\10\end{smallmatrix}}\ar[rddd]\\
&&&&&\color{red}{\begin{smallmatrix}9\\6\\2\end{smallmatrix}}\ar[rd]\ar@[gray] @/_1.6pc/ @{-->}[llluuu]\\
&&&&\color{red}{\begin{smallmatrix}6\\2\end{smallmatrix}}\ar[ru]\ar[rd]\ar@[gray] @/_1.6pc/ @{-->}[llluuu]&&{\begin{smallmatrix}&7&\\3&&9\\&6&\end{smallmatrix}}\ar[rd]\ar[rddd]\ar@[gray] @/^1.6pc/ @{-->}[llluuu]\\
&&&{\begin{smallmatrix}2\end{smallmatrix}}\ar[ru]\ar@[gray] @/^1.6pc/ @{-->}[llluuu]&& {\begin{smallmatrix}3\\6\end{smallmatrix}}\ar[ru]\ar[rddd]\ar@[gray] @/^1.6pc/ @{-->}[llluuu]&& \color{red}{\begin{smallmatrix}4\\7\\9\end{smallmatrix}}\ar[rddd]\ar@[gray] @/_1.6pc/ @{-->}[llluuu]\\
&&&&&\\
&&&&&&& {\begin{smallmatrix}7\\3\end{smallmatrix}}\ar[rd]\ar@[gray] @/_1.6pc/ @{-->}[llluuu]\\
&&&&&& {\begin{smallmatrix}3\end{smallmatrix}}\ar[ru]\ar@[gray] @/^1.6pc/ @{-->}[llluuu]&& \color{red}{\begin{smallmatrix}4\\7\end{smallmatrix}}\ar[rddd]\ar@[gray] @/^1.6pc/ @{-->}[llluuu]\\
\\
\\
&&&&&&&&& {\begin{smallmatrix}4\end{smallmatrix}}\ar@[gray] @/^1.6pc/ @{-->}[llluuu]
}
\end{align*}
\caption{The Auslander-Reiten quiver of $\mathcal{F}$.}
\label{fig:AR_F}
\end{figure}

Consider the full subcategory of $\mathcal{F}$ closed under isomorphisms in $\mathcal{F}$:
\begin{align*}
    \mathcal{X}:=\add\Big\{ \begin{smallmatrix}1\end{smallmatrix},\, \begin{smallmatrix}8\\5\\1\end{smallmatrix},\, \begin{smallmatrix}10\\8\\5\\1\end{smallmatrix},\, \begin{smallmatrix}&9&\\10&&6\\&8&&2\\&&5\end{smallmatrix},\, \begin{smallmatrix}4\\7\\9\\10\end{smallmatrix},\, \begin{smallmatrix}&6&\\2&&8\\&5&\end{smallmatrix},\, \begin{smallmatrix}6\\2\end{smallmatrix},\, \begin{smallmatrix}9\\6\\2\end{smallmatrix},\, \begin{smallmatrix}4\\7\\9\end{smallmatrix},\, \begin{smallmatrix}4\\7\end{smallmatrix} \Big\},
\end{align*}
\textbf{i.e.} $\add$ of the vertices coloured red in Figure \ref{fig:AR_F}. Using
the following module in $\mathcal{X}$:
\begin{align*}
s:=\begin{smallmatrix}1\end{smallmatrix}\oplus \begin{smallmatrix}8\\5\\1\end{smallmatrix}\oplus \begin{smallmatrix}10\\8\\5\\1\end{smallmatrix}\oplus \begin{smallmatrix}&9&\\10&&6\\&8&&2\\&&5\end{smallmatrix}\oplus \begin{smallmatrix}4\\7\\9\\10\end{smallmatrix}\oplus \begin{smallmatrix}&6&\\2&&8\\&5&\end{smallmatrix},
\end{align*}
and let $\Gamma:=\End_{\Phi}(s)$. We check that the conditions (i)-(iv) from \cite[Theorem B]{HJV} hold.

(i) Since $\Phi$ has finite global dimension, then $s$ has finite projective dimension.

(ii) As $s$ is projective in $\mmod\Phi$, it follows that $\Ext^{\geq 1}_{\Phi}(s,s)=0$.

(iii) When $x\in\mathcal{X}$ is a direct summand of $s$, we have a trivial exact sequence. Moreover, we have the following exact sequences:
\begin{align*}
    \xymatrix@R=1em{
    0\ar[r]&{\begin{smallmatrix}1\end{smallmatrix}}\ar[r]& {\begin{smallmatrix}8\\5\\1\end{smallmatrix}}\ar[r]&{\begin{smallmatrix}&6&\\2&&8\\&5&\end{smallmatrix}}\ar[r]&{\begin{smallmatrix}6\\2\end{smallmatrix}}\ar[r]&0,\\
    0\ar[r]&{\begin{smallmatrix}1\end{smallmatrix}}\ar[r]& {\begin{smallmatrix}10\\8\\5\\1\end{smallmatrix}}\ar[r]&{\begin{smallmatrix}&9&\\10&&6\\&8&&2\\&&5\end{smallmatrix}}\ar[r]&{\begin{smallmatrix}9\\6\\2\end{smallmatrix}}\ar[r]&0,\\
    0\ar[r]&{\begin{smallmatrix}8\\5\\1\end{smallmatrix}}\ar[r]& {\begin{smallmatrix}10\\8\\5\\1\end{smallmatrix}}\ar[r]&{\begin{smallmatrix}4\\7\\9\\10\end{smallmatrix}}\ar[r]&{\begin{smallmatrix}4\\7\\9\end{smallmatrix}}\ar[r]&0,\\
    0\ar[r]&{\begin{smallmatrix}&6&\\2&&8\\&5&\end{smallmatrix}}\ar[r]& {\begin{smallmatrix}&9&\\10&&6\\&8&&2\\&&5\end{smallmatrix}}\ar[r]&{\begin{smallmatrix}4\\7\\9\\10\end{smallmatrix}}\ar[r]&{\begin{smallmatrix}4\\7\end{smallmatrix}}\ar[r]&0,\\
    }
\end{align*}
so (iii) is satisfied.

(iv) Consider $\mathcal{G}:=\Hom_{\Phi}(s,\mathcal{X})\subseteq\mmod\Gamma$. In addition to the idempotents in $\Gamma$ corresponding to the identity morphisms, we have the following non-zero morphisms between indecomposable direct summands of $s$:
\begin{align*}
    &\alpha: {\begin{smallmatrix}&9&\\10&&6\\&8&&2\\&&5\end{smallmatrix}}\rightarrow {\begin{smallmatrix}4\\7\\9\\10\end{smallmatrix}},\,\,\,
    \beta: {\begin{smallmatrix}&6&\\2&&8\\&5&\end{smallmatrix}}\rightarrow {\begin{smallmatrix}&9&\\10&&6\\&8&&2\\&&5\end{smallmatrix}},\,\,\,
    \gamma: {\begin{smallmatrix}8\\5\\1\end{smallmatrix}}\rightarrow{\begin{smallmatrix}&6&\\2&&8\\&5&\end{smallmatrix}},\,\,\,
    \delta: {\begin{smallmatrix}1\end{smallmatrix}}\rightarrow {\begin{smallmatrix}8\\5\\1\end{smallmatrix}},\,\,\, \\
    &\epsilon: {\begin{smallmatrix}10\\8\\5\\1\end{smallmatrix}}\rightarrow{\begin{smallmatrix}&9&\\10&&6\\&8&&2\\&&5\end{smallmatrix}},\,\,\,
    \zeta: {\begin{smallmatrix}8\\5\\1\end{smallmatrix}}\rightarrow{\begin{smallmatrix}10\\8\\5\\1\end{smallmatrix}},\,\,\,
    \beta\gamma=\epsilon\zeta: {\begin{smallmatrix}8\\5\\1\end{smallmatrix}}\rightarrow {\begin{smallmatrix}&9&\\10&&6\\&8&&2\\&&5\end{smallmatrix}},
\end{align*}
with $\alpha\beta=0$ and $\gamma\delta=0$.
Then, using \cite[Theorem 3.7, Chapter II]{A}, we have that $\Gamma$ is isomorphic to the algebra $\Psi$ defined by the following quiver with relations:

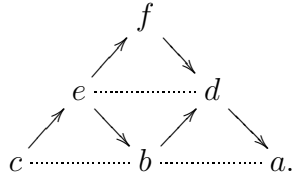
\begin{align*}
    \xymatrix @C=1em@R=1em{
    && f\ar[rd]\\
    & e\ar[ru]\ar[rd]\ar@{..}[rr]&& d\ar[rd]\\
    c\ar[ru]\ar@{..}[rr]&& b\ar[ru]\ar@{..}[rr]&& a.
    }
\end{align*}
We look at $\mathcal{G}$ in terms of quiver representations, using \cite[Theorem 1.6, Chapter III]{A}. So for example, using again the radical series notation, we have
\begin{align*}
    \Hom_\Phi\Big( s, {\begin{smallmatrix}&9&\\10&&6\\&8&&2\\&&5\end{smallmatrix}} \Big)= \begin{smallmatrix}&e&\\b&&f\\&d&\end{smallmatrix}.
\end{align*}
Similarly, we find the radical series of $\Hom_{\Phi}(s,x)$ for each indecomposable $x\in\mathcal{X}$.
Then, using these, it is easy to see that the Auslander-Reiten quiver of $\mathcal{G}$ is as shown in Figure \ref{fig:AR_G}.
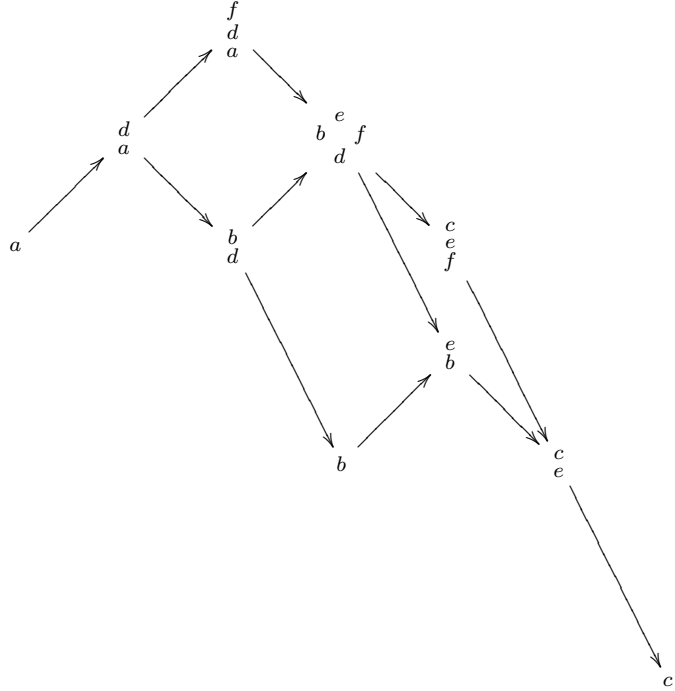
\begin{figure}[h]
\begin{align*}
\xymatrix@!0 @R=3.5em @C=3.5em{
&&{\begin{smallmatrix}f\\d\\a\end{smallmatrix}}\ar[rd]\\
&{\begin{smallmatrix}d\\a\end{smallmatrix}}\ar[ru]\ar[rd]&&{\begin{smallmatrix}&e&\\b&&f\\&d&\end{smallmatrix}}\ar[rd]\ar[rdd]\\
{\begin{smallmatrix}a\end{smallmatrix}}\ar[ru]&& {\begin{smallmatrix}b\\d\end{smallmatrix}}\ar[ru]\ar[rdd]&& {\begin{smallmatrix}c\\e\\f\end{smallmatrix}}\ar[rdd]\\
&&&& {\begin{smallmatrix}e\\b\end{smallmatrix}}\ar[rd]\\
&&& {\begin{smallmatrix}b\end{smallmatrix}}\ar[ru]&& {\begin{smallmatrix}c\\e\end{smallmatrix}}\ar[rdd]\\
\\
&&&&&& {\begin{smallmatrix}c\end{smallmatrix}}
}
\end{align*}
\caption{The Auslander-Reiten quiver of $\mathcal{G}$.}
\label{fig:AR_G}
\end{figure}
By \cite[Remark B.5]{JL}, we conclude that $\mathcal{G}$ is isomorphic to the unique $2$-cluster tilting subcategory of $\mmod\Psi$. Hence $\mathcal{G}\subseteq\mmod\Gamma$ is $2$-cluster tilting.

So (i)-(iv) from \cite[Theorem B]{HJV} hold and we have that $\mathcal{X}$ is a wide subcategory of $\mathcal{F}$ in the sense of \cite[Definition 2.11]{HJV}. In particular, $\mathcal{X}\subseteq \mathcal{F}$ is an additive subcategory closed under $2$-extensions.

Looking at the Auslander-Reiten quiver of $\mathcal{F}$, we see that the following are the $2$-Auslander-Reiten sequences in $\mathcal{F}$ with right end term in $\mathcal{X}$:
\begin{align*}
\xymatrix@R=1em{
0\ar[r]&{\begin{smallmatrix}8\\5\\1\end{smallmatrix}}\ar[r]& {\begin{smallmatrix}&6&\\2&&8\\&5&\end{smallmatrix}}\oplus{\begin{smallmatrix}10\\8\\5\\1\end{smallmatrix}}\ar[r]&{\begin{smallmatrix}6\\2\end{smallmatrix}}\oplus{\begin{smallmatrix}&9&\\10&&6\\&8&&2\\&&5\end{smallmatrix}}\ar[r]&{\begin{smallmatrix}9\\6\\2\end{smallmatrix}}\ar[r]&0&& \text{(a)}\\
0\ar[r]&{\begin{smallmatrix}5\\1\end{smallmatrix}}\ar[r]& {\begin{smallmatrix}2\\5\end{smallmatrix}}\oplus{\begin{smallmatrix}8\\5\\1\end{smallmatrix}}\ar[r]&{\begin{smallmatrix}2\end{smallmatrix}}\oplus{\begin{smallmatrix}&6&\\2&&8\\&5&\end{smallmatrix}}\ar[r]&{\begin{smallmatrix}6\\2\end{smallmatrix}}\ar[r]&0,&& \text{(b)}\\
0\ar[r]&{\begin{smallmatrix}3\\6\\8\end{smallmatrix}}\ar[r]& {\begin{smallmatrix}3\\6\end{smallmatrix}}\oplus{\begin{smallmatrix}&7&\\3&&9\\&6&&10\\&&8\end{smallmatrix}}\ar[r]&{\begin{smallmatrix}&7\\3&&9\\&6\end{smallmatrix}}\oplus{\begin{smallmatrix}4\\7\\9\\10\end{smallmatrix}}\ar[r]&{\begin{smallmatrix}4\\7\\9\end{smallmatrix}}\ar[r]&0,&& \text{(c)}\\
0\ar[r]&{\begin{smallmatrix}3\\6\end{smallmatrix}}\ar[r]& {\begin{smallmatrix}3\end{smallmatrix}}\oplus{\begin{smallmatrix}&7\\3&&9\\&6\end{smallmatrix}}\ar[r]&{\begin{smallmatrix}7\\3\end{smallmatrix}}\oplus{\begin{smallmatrix}4\\7\\9\end{smallmatrix}}\ar[r]&{\begin{smallmatrix}4\\7\end{smallmatrix}}\ar[r]&0.&& \text{(d)}
}
\end{align*}

Note that all the terms in (a) are in $\mathcal{X}$, so (a) is also a $2$-Auslander-Reiten sequence in $\mathcal{X}$. Moreover, the following are $\mathcal{X}$-covers:
\begin{align*}
{\begin{smallmatrix}1\end{smallmatrix}}\rightarrow {\begin{smallmatrix}5\\1\end{smallmatrix}},\,\,\,
{\begin{smallmatrix}&6\\2&&8\\&5\end{smallmatrix}}\rightarrow {\begin{smallmatrix}3\\6\\8\end{smallmatrix}},\,\,\,
{\begin{smallmatrix}6\\2\end{smallmatrix}}\rightarrow {\begin{smallmatrix}3\\6\end{smallmatrix}}.
\end{align*}
Then, using these covers, (b), (c), (d), Theorem \ref{coro_d-ARseq} and the fact that the relevant $\Ext^2$-spaces are one dimensional by \cite[Theorem 3.6]{OT}, we find the $2$-Auslander-Reiten sequences in $\mathcal{X}$:

\begin{align*}
\xymatrix@R=1em{
0\ar[r]&{\begin{smallmatrix}1\end{smallmatrix}}\ar[r]& {\begin{smallmatrix}8\\5\\1\end{smallmatrix}}\ar[r]&{\begin{smallmatrix}&6&\\2&&8\\&5&\end{smallmatrix}}\ar[r]&{\begin{smallmatrix}6\\2\end{smallmatrix}}\ar[r]&0,&& \text{(b')}\\
0\ar[r]&{\begin{smallmatrix}&6\\2&&8\\&5\end{smallmatrix}}\ar[r]& {\begin{smallmatrix}&9&\\10&&6\\&8&&2\\&&5\end{smallmatrix}}\oplus{\begin{smallmatrix}6\\2\end{smallmatrix}}\ar[r]&{\begin{smallmatrix}9\\6\\2\end{smallmatrix}}\oplus{\begin{smallmatrix}4\\7\\9\\10\end{smallmatrix}}\ar[r]&{\begin{smallmatrix}4\\7\\9\end{smallmatrix}}\ar[r]&0,&& \text{(c')}\\
0\ar[r]&{\begin{smallmatrix}6\\2\end{smallmatrix}}\ar[r]& {\begin{smallmatrix}9\\6\\2\end{smallmatrix}}\ar[r]&{\begin{smallmatrix}4\\7\\9\end{smallmatrix}}\ar[r]&{\begin{smallmatrix}4\\7\end{smallmatrix}}\ar[r]&0.&& \text{(d')}
}
\end{align*}

\end{document}